\documentclass{amsart}
\usepackage[all]{xy}
\usepackage{amssymb,latexsym, amscd, amsthm, amsmath, upgreek, bm}
\newtheorem{theorem}{Theorem}[section]
\newtheorem{lemma}[theorem]{Lemma}
\newtheorem{corollary}[theorem]{Corollary}
\newtheorem{proposition}[theorem]{Proposition}
\newtheorem{conjecture}[theorem]{Conjecture}

\theoremstyle{definition}
\newtheorem{definition}[theorem]{Definition}

\theoremstyle{remark}
\newtheorem{remark}[theorem]{\bf Remark}

\newcommand{\Q}{\mathbb Q}

\newcommand{\Z}{\mathbb Z}

\newcommand{\N}{\mathbb N}
\newcommand{\F}{\mathbb F}

\newcommand{\blambda}{{\boldsymbol{\lambda}}}

\begin{document}

\title[Stickelberger splitting in the $K$--theory of number fields]{{The Stickelberger splitting map and Euler systems in the $K$--theory of number fields}}

\author[G. Banaszak]{Grzegorz Banaszak*}
\address{Department of Mathematics and Computer Science, Adam Mickiewicz University,
Pozna\'{n} 61614, Poland}
\email{banaszak@amu.edu.pl}

\author[C. D. Popescu]{Cristian D. Popescu**}
\address{Department of Mathematics, University of California, San Diego, La Jolla, CA 92093, USA}
\email{cpopescu@math.ucsd.edu}

\date{}

\thanks{*Partially supported by
grant NN201607440 of the Polish Ministry of Science
and Education. \newline
\indent **Partially supported by NSF grants
DMS-901447 and DMS-0600905}

\begin{abstract}
For a CM abelian extension $F/K$ of an arbitrary totally real
number field $K$, we construct the Stickelberger splitting maps
(in the sense of \cite{Ba1}) for both the \'etale and the Quillen
$K$--theory of $F$ and we use these maps to construct Euler
systems in the even Quillen $K$--theory of $F$. The Stickelberger
splitting maps give an immediate proof of the annihilation of the
groups of divisible elements  $div K_{2n}(F)_l$ of the even
$K$--theory of the top field by higher Stickelberger elements, for
all odd primes $l$. This generalizes the results of \cite{Ba1},
which only deals with CM abelian extensions of $\Bbb Q$. The
techniques involved in constructing our Euler systems at this
level of generality are quite different from those used in
\cite{BG1}, where an Euler system in the odd $K$--theory with
finite coefficients of abelian CM extensions of $\Bbb Q$ was
given. We work under the assumption that the Iwasawa
$\mu$--invariant conjecture holds. This permits us to make use of
the recent results of Greither-Popescu \cite{GP} on the \'etale
Coates-Sinnott conjecture for arbitrary abelian extensions of
totally real number fields, which are conditional upon this
assumption. In upcoming work, we will use the Euler systems
constructed in this paper to obtain information on the groups of
divisible elements $div K_{2n}(F)_l$, for all $n>0$ and odd $l$.
It is known that the structure of these groups is intimately
related to some of the deepest unsolved problems in algebraic
number theory, e.g. the Kummer-Vandiver and Iwasawa conjectures on
class groups of cyclotomic fields. We make these connections
explicit in the introduction.
\end{abstract}

\maketitle


\section{Introduction}

 Let $F / K$ be an abelian CM extension of a totally real
number field $K.$ Let ${\bf f}$ be the conductor of $F / K$ and
let $K_{{\bf f}} / K$ be the ray--class field extension with
conductor ${\bf f}.$ Let $G_{\bf f} := G(K_{{\bf f}} / K).$ For
all $n\in\Bbb Z_{\geq 0}$, Coates \cite{C} defined higher
Stickelberger elements $\Theta_{n} ({\bf b}, {\bf f}) \in \Q [G(F
/ K)]$, for integral ideals $\bf b$ of $K$ coprime to $\bf f$.
Deligne and Ribet \cite{DR} proved that $\Theta_{n} ({\bf b}, {\bf
f}) \in \Z[G(F / K)]$, if ${\bf b}$ is also coprime to
$w_{n+1}(F):={\rm card}\, H^0(F, \Bbb Q/\Bbb Z(n+1))$.  A detailed
discussion of the Stickelberger elements and
their basic properties is given in \S2 below. In 1974, Coates and Sinnott \cite{CS}
formulated the following conjecture.

\begin{conjecture}[Coates-Sinnott]\label{Conj. Coates-Sinnott}
For all $n\geq 1$ and all ${\bf b}$ coprime to $w_{n+1}(F)$, $\Theta_{n} ({\bf b}, {\bf f})$
annihilates $K_{2n} ({\mathcal O}_F)$.
\end{conjecture}

\noindent This should be viewed as a higher analogue of the classical conjecture of Brumer.

\begin{conjecture}[Brumer]\label{Conj. Brumer} For all ${\bf b}$ coprime to $w_1(F)$, $\Theta_{0} ({\bf b}, {\bf f})$
annihilates $K_{0} ({\mathcal O}_F)_{\rm tors}=Cl(O_F).$
\end{conjecture}

\noindent Coates and Sinnott \cite{CS} proved that for the base field $K =
\Q$ the element $\Theta_{1} ({\bf b}, {\bf f})$ annihilates $K_{2}
({\mathcal O}_F)$ for $F/\Q$ abelian and ${\bf b}$ coprime to the
order of $K_{2} ({\mathcal O}_F).$ Moreover, in the case $K=\Bbb Q$, they proved that
$\Theta_{n} ({\bf b}, {\bf f})$ annihilates the $l$--adic \'etale cohomology groups $H^2(\mathcal O_F[1/l], \Bbb Z_l(n+1))\simeq K_{2n}^{et}
({\mathcal O}_F[1/l])$ for any odd prime $l$, and any odd $n\geq 1$.  One of the ingredients used in the proof is the fact
that Brumer's conjecture holds true if $K=\Bbb Q$. This is the classical theorem of Stickelberger.
The passage from annihilation of
\'etale cohomology to that of $K$--theory in the case $n=1$ was possible due to the following theorem (see \cite{Ta2}, \cite{Co1} and \cite{Co2}.)

\begin{theorem}[Tate]\label{Tate's Theorem}  The $l$--adic Chern map gives a canonical isomorphism
$${K_{2} ({\mathcal O}_L) \otimes \Z_l
\stackrel{\cong}{\longrightarrow}
 K_{2}^{et} ({\mathcal O}_L[1/l])},
$$
for any number field $L$ and any odd prime $l$.
\end{theorem}

\noindent The following deep conjecture aims at generalizing Tate's theorem.

\begin{conjecture}[Quillen-Lichtenbaum]\label{Conj. Quillen-Lichtenbaum}
For any number field $L$, any $m \geq 1$ and any odd prime $l$ there is a natural
$l$--adic Chern map isomorphism
\begin{eqnarray}{K_{m} ({\mathcal O}_L) \otimes \Z_l
\stackrel{\cong}{\longrightarrow}
 K_{m}^{et} ({\mathcal O}_L[1/l])}
\end{eqnarray}
\end{conjecture}
\medskip

Very recently, Greither and the second author used Iwasawa theoretic techniques to prove the following results for a general abelian CM extension
$F/K$ of an arbitrary totally real field $K$ (see \cite{GP}.)

\begin{theorem}[Greither-Popescu]\label{Greither-Popescu} Assume that $l$ is odd and the Iwasawa $\mu$--invariant $\mu_{F,l}$ associated to
$F$ and $l$ vanishes. Then, we have the following.
\begin{enumerate}
\item $\prod_{\bf l}(1-({\bf l}, F/K)^{-1}\cdot N{\bf l})\cdot \Theta_0({\bf b}, {\bf f})$ annihilates $Cl(O_F)_l$, for all ${\bf b}$ coprime to $w_1(K)l$, where the product is taken
over primes ${\bf l}$ of $K$ which divide $l$ and are coprime to ${\bf f}$.
\item $\Theta_n({\bf b}, {\bf f})$ annihilates $K_{2n}^{et}(\mathcal O_F[1/l])$, for all $n\geq 1$ and all ${\bf b}$ coprime to $w_{n+1}(F)l$.
\end{enumerate}

\end{theorem}

\noindent In fact, stronger results are proved in \cite{GP}, involving Fitting ideals rather than annihilators
and, in the case $n=0$,  a refinement of Brumer's conjecture, known as the Brumer-Stark conjecture (see Theorems 6.5 and 6.11 in loc.cit.)
\medskip

\noindent Results similar to the Fitting ideal version of part (2) of Theorem \ref{Greither-Popescu} were also obtained with different methods by Burns--Greither in \cite{Burns-Greither} and
by Nguyen Quang Do in \cite{NQD}, under some extra hypotheses.
\medskip

\noindent Note that a well known conjecture of Iwasawa states that $\mu_{F,l}=0$, for all $l$ and $F$ as above. This conjecture is known to
hold if $F$ is an abelian extension of $\Bbb Q$, due independently to
Ferrero-Washington and Sinnott. Consequently, if the Quillen-Lichtenbaum conjecture is proved,
then, for all odd primes $l$, the $l$--primary
part of the Coates-Sinnott conjecture is established unconditionally for all abelian extensions $F/\Q$ and for general extensions $F/K$, under the assumption that
$\mu_{F,l}=0$.
It is hoped that
recent work of Suslin, Voyevodsky, Rost,
Friedlander, Morel, Levine, Weibel and others will lead to a proof of
the Quillen-Lichtenbaum conjecture.
\bigskip

In 1992, a different approach towards the Coates-Sinnott conjecture was
used in \cite{Ba1}, in the case $K = \Q.$ Namely, for all $n\geq 1$, all ${\bf b}$ coprime to $w_{n+1}(F)$,
and $l>2$, the first author constructed in Ch. IV of loc.cit. the Stickelberger
splitting map $\Lambda:=\Lambda_n$ of the boundary map $\partial_F$ in the
Quillen localization sequence
$$ 0 \stackrel{}{\longrightarrow} K_{2n} ({\mathcal O}_{F})_l
\stackrel{}{\longrightarrow} K_{2n} (F)_l
{{\stackrel{\partial_F}{\longrightarrow}} \atop
{{\stackrel{\Lambda}{\longleftarrow}}}} \bigoplus_{v} K_{2n-1}
(k_v)_l \stackrel{}{\longrightarrow} 0.$$ By definition, $\Lambda$ is a homomorphism such that $\partial_F
\circ \Lambda$ is the multiplication by $\Theta_{n} ({\bf b},
{\bf f}).$ Above, $k_v$ denotes the residue field of a prime $v$ in $O_F$.

The existence of such a map $\Lambda$ implies that $\Theta_{n} ({\bf b}, {\bf
f})$ annihilates the group $div (K_{2n} (F)_l)$ of divisible elements
in $K_{2n}(F)_l$ (see loc.cit. as well as Theorem \ref{Theorem 7.13} below.)  This group is contained
in $K_{2n} ({\mathcal O}_F)_l$, which is obvious from the exact sequence above and
the finiteness of  $K_{2n-1}
(k_v)_l$, for all $v$.

The construction of $\Lambda$ in loc.cit. was
done without appealing to \'etale cohomology and the
Quillen-Lichtenbaum conjecture. However, it was based on the fact that Brumer's Conjecture is known to
hold for abelian extensions of $\Q$ (Stickelberger's
theorem). Since Brumer's conjecture was not yet proved
over arbitrary totally real base fields (and it is still not proved unconditionally at that level of generality), the construction of
$\Lambda$ in loc.cit. could not be generalized. Also, it should be mentioned that in loc.cit.
various technical difficulties arose at primes $l|n$ and the map $\Lambda$ was constructed only up to a certain power
$l^{v_l(n)}$ in those cases.

In 1996, in joint work with Gajda \cite{BG1}, the first author discovered a new, perhaps
deeper and farther reaching application of the existence of $\Lambda$ for abelian extensions $F/\Bbb Q$.
Namely, $\Lambda$ was used in \cite{BG1} to construct special elements which give rise to Euler systems in the
$K$--theory with finite coefficients  $\{K_{2n+1}(L, \Bbb Z/l^k)\}_L$, where $L$ runs over all abelian extensions
of $\Bbb Q$, such that $F\subseteq L$ and $L/F$ has a square-free conductor coprime to ${\bf f}l$. Now, it is hoped that
these Euler systems can be used to study the structure of the group of
divisible elements $div K_{2n}(F)_l$, for all $n\geq 1$. This is a goal truly worth pursuing, as this group structure
is linked to some of the deepest unsolved problems in algebraic number theory, as shown at the end of this introduction.
\bigskip

The main goal of this paper is to generalize the results obtained in \cite{Ba1} and \cite{BG1} to the case of
CM abelian extensions $F/K$ of arbitrary totally real number fields $K$. Moreover, in terms of constructing Euler systems,
we go far beyond \cite{BG1} in that we construct Euler systems in Quillen $K$--theory rather than
$K$--theory with finite coefficients only.
Roughly speaking, our strategy is as follows.

{\bf Step 1.} We fix an integer $m>0$ and assume that the $m$--th
Stickelberger elements $\Theta_{m} ({\bf b}, {\bf f}_k)$
annihilate $K_{2m} ({\mathcal O}_{F_k})_l$ (respectively
$K^{et}_{2m} ({\mathcal O}_{F_k})_l$) for each $k$, where $F_k :=
F(\mu_{l^k})$ and ${\bf f}_k$ is the conductor of $F_k/K$. Under
this assumption, we construct the Stickelberger splitting maps
$\Lambda_m$ (respectively $\Lambda_m^{et}$) for the $K$--theory
(respectively \'etale $K$--theory) of $F_k$, for all $k\geq 1$.
(See Lemma \ref{Lemma 7.4} and the constructions which lead to
it.) Note that, if combined with Theorem \ref{Tate's Theorem},
Theorem \ref{Greither-Popescu} shows that $\Theta_{m} ({\bf b},
{\bf f}_k)$ annihilates $K_{2m} ({\mathcal O}_{F_k})_l$, for $m=1$
and $l$ odd, under the assumption that $\mu_{F,l}=0$ (and
unconditionally if $F/\Bbb Q$ is abelian.) Also, in \cite{Po}, the
first author constructs an infinite class of abelian CM extensions
$F/K$ of an arbitrary totally real number field $K$ for which the
annihilation of $K_{2m}({\mathcal O}_{F_k})_l$  by $\Theta_{m}
({\bf b}, {\bf f}_k)$, for $m=1$ and $l$ odd is proved
unconditionally.

{\bf Step 2.} We use the map $\Lambda_m$ (respectively $\Lambda_m^{et})$) of Step 1 to construct special elements $\blambda_{v, l^k}$
(respectively $\blambda^{et}_{v, l^k}$) in the $K$--theory with coefficients $K_{2n}(O_{F,S_v}; \Bbb Z/l^k)$
(respectively \'etale $K$--theory with coefficients $K_{2n}^{et}(O_{F,S_v}; \Bbb Z/l^k)$), for all $n>0$, all $k\geq 0$ and all primes
$v$ in $O_F$, where $S_v$ is a sufficiently large finite set of primes in $F$.
(See Definition \ref{special-coefficients} .)

{\bf Step 3.} We use the special elements of Step 3 and a projective limit process with respect to $k$
to construct the Stickelberger splitting maps $\Lambda_n$ and $\Lambda^{et}_n$ taking values in
$K_{2n}(F)_l$ and $K_{2n}^{et}(F)$, respectively, for all $n\geq 1.$ (See Definition \ref{Definition 7.9}
and Theorem \ref{Theorem 7.10}.) This step generalizes the constructions in \cite{Ba1} to abelian CM extensions
of arbitrary totally real fields. It also eliminates
the extra-factor $l^{v_l(n)}$ which appeared in loc.cit. in the case $l|n$, for abelian CM extensions of $\Bbb Q$.

{\bf Step 4.} We use the special elements of Step 2 as well as the maps $\Lambda_n$ of Step 3 to construct Euler
Systems $\{\Lambda_n (\xi_{v ({\bf L})})\}_{\bf L}$ in the $K$-theory
without coefficients $\{K_{2n}(F_{\bf L})_l\}_{\bf L}$, for every $n>0$, where
${\bf L}$ runs through the squarefree ideals of $O_F$ which are coprime to ${\bf f}l$, $F_{\bf L}$ is the
ray class field of $F$ corresponding to ${\bf L}$ and $S$ is a sufficiently large finite set of primes in $O_F$.
(See Definitions \ref{ES1-definition} and \ref{ES-definition} as well as Theorem \ref{Theorem ES3}.) A similar construction of Euler systems in \'etale $K$--theory
can be done without difficulty. This step generalizes the constructions
of \cite{BG1} to the case of abelian CM extensions of totally real number fields. It is also worth noting that
while \cite{BG1} contains a construction of Euler systems only in the case of $K$--theory with coefficients,
we deal with both the $K$--theory with and without coefficients in the more general setting discussed in this paper.

 In the process, as a consequence of the construction of $\Lambda_n$ (Step 3),
we obtain a direct proof that $\Theta_n({\bf b}, \bf f)$
annihilates the group $div (K_{2n} (F)_l)$, for arbitrary CM
abelian extensions $F/K$ of totally real base field $K$ and all
$n>0$, under the assumption that $l>2$ and $\mu_{F,l}=0$  (see
Theorem \ref{Theorem 7.18}.)

In our upcoming work, we are planning on using the Euler systems described in Step 4 above to study the structure
of the groups of divisible elements $div K_{2n}(F)_l$, for all $n>0$ and all $l>2$.

\bigskip

We conclude this introduction with a few paragraphs showing that the groups
of divisible elements in the $K$--theory
of number fields lie at the heart of several important conjectures in number theory, which justifies
the effort to understand their structure in terms of special values of global $L$--functions.
In 1988, Warren Sinnott pointed out to the first author that
Stickelberger's Theorem for an abelian extension $F/\Q$
 or, more generally, Brumer's conjecture for a CM extension $F/K$
of a totally real number field  $K$ is equivalent to the existence of
a Stickelberger splitting map $\Lambda$ in the following basic exact sequence
$$
0 \stackrel{}{\longrightarrow} {\mathcal O}_{F}^{\times}
\stackrel{}{\longrightarrow} F^{\times}
{{\stackrel{\partial_F}{\longrightarrow}} \atop
{{\stackrel{\Lambda}{\longleftarrow}}}} \bigoplus_{v} \Z
\stackrel{}{\longrightarrow} Cl(O_F) \stackrel{}{\longrightarrow} 0
.$$
This means that $\Lambda$ is a group homomorphism, such that
$\partial_F \circ \Lambda$ is the multiplication by
$\Theta_{0} ({\bf b}, {\bf f}).$ Obviously, the above exact sequence
is the lower part of the Quillen localization sequence in $K$--theory, since
$K_1 ({\mathcal O}_{F}) = {\mathcal O}_{F}^{\times},$
$K_1 (F) = F^{\times},$ $K_0 (k_v) = \Z,$ $K_0 ({\mathcal O}_{F})_{\rm tors} = Cl(O_F)$
and Quillen's $\partial_F$ is the direct sum of the valuation maps in this case.

Further, by \cite{Ba2} p. 292 we observe that for any prime $l > 2$,
the annihilation of $div (K_{2n} (F)_l)$ by $\Theta_{n} ({\bf b}, {\bf f})$ is equivalent to
the existence of a ``splitting'' map $\Lambda$ in the following exact sequence
$$
0 \stackrel{}{\longrightarrow} K_{2n} ({\mathcal O}_{F}) [l^k]
\stackrel{}{\longrightarrow} K_{2n} (F) [l^k]
{{\stackrel{\partial_F}{\longrightarrow}} \atop
{{\stackrel{\Lambda}{\longleftarrow}}}} \bigoplus_{v} K_{2n - 1} (k_v) [l^k]
\stackrel{}{\longrightarrow} div
(K_{2n} (F)_l) \stackrel{}{\longrightarrow} 0
$$
such that $\partial_F \circ \Lambda$ is the multiplication by
$\Theta_{n} ({\bf b}, {\bf f})$, for any $k\gg 0$.
Hence, the group of divisible elements $div (K_{2n} (F)_l)$ is a direct analogue
of the $l$--primary part $Cl(O_F)_l$ of the class group.
Any two such ``splittings'' $\Lambda$ differ by a homomorphism
in ${\rm Hom} (\bigoplus_{v} K_{2n - 1} (k_v) [l^k], \,\, K_{2n} ({\mathcal O}_{F}) [l^k]).$
Moreover, the Coates-Sinnott conjecture is equivalent to the existence of a ``splitting'' $\Lambda$,
such that $\Lambda \circ \partial_F$ is the multiplication
by $\Theta_{n} ({\bf b}, {\bf f}).$ If the Coates-Sinnott conjecture holds,
then such a ``splitting'' $\Lambda$ is unique and
satisfies the property that $\partial_F \circ \Lambda$ is equal to the multiplication
by $\Theta_{n} ({\bf b}, {\bf f})$. This is due to the fact that $div(K_{2n} (F)_l) \subset
K_{2n} ({\mathcal O}_{F})_l.$ Clearly, in the case $div(K_{2n} (F)_l) =
K_{2n} ({\mathcal O}_{F})_l$, our map $\Lambda$ also has the property that
$\Lambda \circ \partial_F$ equals multiplication
by $\Theta_{n} ({\bf b}, {\bf f}).$ Observe that if the Quillen-Lichtenbaum
conjecture holds, then by Theorem 4 in \cite{Ba2}, we have
$$div(K_{2n} (F)_l) =
K_{2n} ({\mathcal O}_{F})_l\,\, \Leftrightarrow \,\, \left\vert
\frac{\prod_{v | l} w_n (F_v)}{w_n (F)} \right\vert_{l}^{-1} = 1.$$
In particular, for $F = \Q$ and $n$ odd, we have $w_n (\Q) = w_n (\Q_l) = 2$.
Hence, according to the Quillen-Lichtenbaum conjecture, for any $l > 2$
we should have $div (K_{2n} (\Q)_l) = K_{2n} (\Z)_l.$
\medskip

Now, let $A := Cl(\Z [\mu_l])_l$ and let
$A^{[i]}$ denote the eigenspace corresponding  to the $i$--th power of the
Teichmuller character $\omega\, :\, G(\Q(\mu_l) / \Q) \rightarrow (\Z/l\Z)^{\times}.$
Consider the following classical conjectures in cyclotomic field theory.
\begin{conjecture}[Kummer-Vandiver] \label{Conj. Kumer-Vandiver}
\begin{eqnarray}
A^{[l - 1 -n]} = 0 \quad \text{for all n even and} \quad 0 \leq n \leq l-1
\nonumber
\end{eqnarray}
\end{conjecture}

\begin{conjecture}[Iwasawa] \label{Conj. Iwasawa}
\begin{eqnarray}
A^{[l - 1 -n]} \quad \text{is cyclic for all $n$ odd, such that} \quad
1 \leq n \leq l-2
\nonumber
\end{eqnarray}
\end{conjecture}

\noindent We can state the Kummer-Vandiver and Iwasawa
conjectures in terms of divisible elements in $K$--theory of $\Q$ (see \cite{BG1} and \cite{BG2}):
\begin{itemize}
\item[(1)]
$A^{[l - 1 -n]} = 0$ $\Leftrightarrow$ $div (K_{2n} (\Q)_l) = 0,\text{ for all } n$
\text{even, with} $1\leq n \leq (l-1).$
\item[(2)]
$A^{[l - 1 -n]}$ \text{is cyclic} $\Leftrightarrow$
$div (K_{2n} (\Q)_l)$\text{ is cyclic, for all $n$
odd, with} $n \leq (l-2).$
\end{itemize}
\medskip

Finally, we would like to point out that the groups of divisible elements discussed in this paper are also related
to the Quillen-Lichtenbaum conjecture.
Namely, by comparing the exact sequence
of \cite{Sch}, Satz 8 with the exact sequence of \cite{Ba2}, Theorem 2 we conclude
that the Quillen-Lichtenbaum conjecture
for the $K$-group $K_{2n} (F)$ (for any number field $F$ and any prime $l > 2$)
holds if and only if
$$div (K_{2n} (F)_l) = K_{2n}^{w} ({\mathcal O}_{F})_l$$
where $K_{2n}^{w} ({\mathcal O}_{F})_l$ is the wild kernel defined in
\cite{Ba2}.


\noindent
\section{Basic facts about the Stickelberger ideals}
Let $F / K$ be an abelian CM extension of a totally real number field
$K.$ Let ${\bf f}$ be the conductor of $F / K$ and let $K_{{\bf
f}} / K$ be the ray class field extension corresponding to ${\bf
f}.$ Let $G_{\bf f} := G(K_{{\bf f}} / K).$ Every element of
$G_{\bf f}$ is the Frobenius morphism $\sigma_{{\bf a}}$, for some
ideal $\bf a$ of ${\mathcal O}_K$, coprime to the conductor $\bf
f$. Let $({\bf a}, F)$ denote the image of $\sigma_{{\bf a}}$ in
$G(F / K)$ via the natural surjection $G_{\bf f} \rightarrow G(F /
K).$ Choose a prime number $l.$
\medskip

\noindent With the usual notations, we let $I({\bf f}) / P_{1}
({\bf f})$ be the ray class group of fractional ideals in $K$
coprime to ${\bf f}.$ Let ${\bf a}$ and ${\bf a}^{\prime}$ be two
fractional ideals in $I({\bf f}).$ The symbol ${\bf a} \equiv {\bf
a}^{\prime}  \mod {\bf f}$ will mean that ${\bf a}$ and ${\bf
a}^{\prime}$ are in the same class modulo $P_{1} ({\bf f}).$ For every ${\bf a}\in I({\bf f})$
we consider the partial zeta function of \cite{C},
p. 291, given by
\begin{equation}
\zeta_{{\bf f}} ({\bf a}, s) := \sum_{{\bf c} \equiv {\bf a}  \,
\rm{mod} \,  {\bf f} } \,\, {\frac{1} {N{\bf c}^s}}\,,\qquad {\rm Re}(s)>1,
\label{2.0}\end{equation} where the sum is taken over the integral
ideals $\bf c\in I(\bf f)$ and $N\bf c$ denotes the usual norm of
the integral ideal $\bf c$. \noindent The partial zeta
$\zeta_{{\bf f}} ({\bf a}, s)$ can be meromorphically continued to
the complex plane with a single pole at $s = 1.$ For $s\in\Bbb
C\setminus\{1\}$, consider the Sickelberger element of [C], p.
297,
\begin{equation}
\Theta_{s} ({\bf b}, {\bf f}) := (N {\bf b}^{s+1} - ({\bf b}, \,
F)) \sum_{\bf a}\, \zeta_{{\bf f}} ({\bf a}, -s) ({\bf a},\,
F)^{-1}\in \Bbb C[G(F/K)] \label{2.1}\end{equation} where ${\bf b}$ is an integral ideal in $I(\bf f)$ and the summation is over a
finite set $\mathcal S$ of ideals ${\bf a}$ of ${\mathcal O}_{K}$
coprime to ${\bf f}$, chosen such that the Artin map
$$\mathcal S\longrightarrow G(K_{\bf f} /
K)\,,\quad \bf a\longrightarrow \sigma_{\bf a}$$ is bijective. The
element $\Theta_{s} ({\bf b}, {\bf f})$ can be written in the
following way

\begin{equation}
\Theta_{s} ({\bf b}, {\bf f}) :=
\sum_{\bf a}\, \Delta_{s+1} ({\bf a}, {\bf b}, {\bf f} ) ({\bf a} ,\, F)^{-1},
\label{2.7}\end{equation}
where
\begin{equation}
\Delta_{s+1} ({\bf a} , {\bf b} , {\bf f} ) :=
N {\bf b} ^{s+1} \zeta_{{\bf f}} ({\bf a} , -s) -
\zeta_{{\bf f}} ({\bf a} {\bf b} , -s).
\label{2.8}\end{equation}

Arithmetically, the Stickelberger elements $\Theta_{s} ({\bf b},
{\bf f})$ are most interesting for values $s = n$, with $n\in \N \cup
\{0\}.$ If ${\bf a}, {\bf b}, {\bf f}$ are integral ideals, such
that ${\bf a} {\bf b}$ is coprime to ${\bf f}$,  then Deligne and
Ribet \cite{DR} proved that $\Delta_{n+1} ({\bf a}, {\bf b}, {\bf f})$
are $l$-adic integers for all primes $l \not |\, N {\bf b}$ and
all $n \geq 0$. Moreover, in loc.cit. it is proved that
\begin{equation}
\Delta_{n+1} ({\bf a} , {\bf b} , {\bf f} ) \equiv N ({\bf a} {\bf
b} )^n \Delta_{1} ({\bf a}, {\bf b}, {\bf f}) \mod w_{n}(K_{\bf
f}). \label{2.9}\end{equation} As usual, if $L$ is a number field,
then $w_{n}(L)$ is the largest number $m \in \N$ such that the
Galois group $G(L (\mu_m) / L)$ has exponent dividing $n.$ Note
that $$w_{n}(L) = |H^0 (G(\overline{L}/L), \, \Q / \Z(n))|\,,$$
where $\Q /\Z (n) := \oplus_{l} \Q_l / \Z_l (n).$ By Theorem 2.4
of [C],  the results in \cite{DR} lead to
$$\Theta_{n} ({\bf b}, {\bf f}) \in \Z[G(F/ K)],$$ whenever
${\bf b}$ is coprime to $w_{n+1} (F).$ The ideal of $\Z [G (F /
K)]$ generated by the elements $\Theta_{n} ({\bf b}, {\bf f})$,
for all integral ideals ${\bf b}$ coprime to $w_{n+1} (F)$ is
called the $n$-th Stickelberger ideal for $F/K.$

\medskip
When $K \subset F \subset E$ is a tower of finite abelian
extensions then $$Res_{E/F} \, : \, G(E/K) \rightarrow G(F/K),\qquad Res_{E/F}: \Bbb C[G(E/K)]\rightarrow \Bbb C[G(F/K)]$$
denote the restriction map and its $\Bbb C$--linear extension at the level of group rings, respectively.
If ${\bf f}\, | \, {\bf f}^{\prime}$ and ${\bf f}$ and ${\bf
f}^{\prime}$ are divisible by the same prime ideals of ${\mathcal
O}_K$ then,  for all ${\bf b}$ coprime to ${\bf f}$,  we have the
following equality (see \cite{C} Lemma 2.1, p. 292).

\begin{equation}
Res_{K_{{\bf f}^{\prime}}/K_{{\bf f}}} \,
\Theta_{s} ({\bf b}, {\bf f}^{\prime}) =
\Theta_{s} ({\bf b}, {\bf f}).
\label{2.2}\end{equation}
\medskip

\noindent Let ${\bf l}$ is a prime ideal of $O_K$ coprime to ${\bf
f}$. Then, we have
\begin{equation}
\zeta_{{\bf f}} ({\bf a}, s) :=
\sum_{{{\bf c} \equiv {\bf a}  \, \rm{mod} \,
{\bf f}} \atop {{\bf l}\, \nmid \, {\bf c}} }
\,\, {\frac{1}{N{\bf c}^s}} +
\sum_{{{\bf c} \equiv {\bf a} \, \rm{mod} \,  {\bf f}} \atop
{{\bf l}\, | \, {\bf c}}} \,\, {\frac{1}{N{\bf c}^s}}.
\label{2.3}\end{equation}
\noindent
Observe that we also have
\begin{equation}
\sum_{{{\bf c} \equiv {\bf a}  \, \rm{mod} \,  {\bf f}} \atop
{{\bf l}\, \nmid \, {\bf c}} } \,\, {\frac{1}{N{\bf c}^s}} =
\sum_{{{\bf a}^{\prime} \, \rm{mod} \,
{\bf l f} } \atop { {\bf a}^{\prime} \equiv {\bf a}
 \, \rm{mod} \,  {\bf f}} }  \sum_{{\bf c} \equiv {\bf a}^{\prime}
 \, \rm{mod} \,  {\bf l f} }
\,\, {\frac{1}{N{\bf c}^s}} =
\sum_{{{\bf a}^{\prime} \, \rm{mod} \,  {\bf l f} } \atop { {\bf a}^{\prime}
\equiv {\bf a}  \, \rm{mod} \,  {\bf f} }}
\zeta_{{\bf l f}} ({\bf a}^{\prime}, s)
\label{2.4}\end{equation}
\medskip

\noindent Let us fix a finite $\mathcal S$ of integral ideals
${\bf a}$ in $I(\bf f)$ as above. Observe that every class
corresponding to an integral ideal ${\bf a}$ modulo $P_{1} ({\bf
f})$ can be written uniquely as a class ${\bf l}{\bf
a}^{\prime\prime}$ modulo $P_{1} ({\bf f})$,  for some ${\bf
a}^{\prime\prime}$ from our set $\mathcal S$ of chosen integral
ideals. This
establishes a one--to--one correspondence between classes ${\bf
a}$ modulo $P_{1} ({\bf f})$ and ${\bf a}^{\prime\prime}$ modulo
$P_{1} ({\bf f}).$ If ${\bf l}\, | \,{\bf c}$,  we put ${\bf c} =
{\bf l} {\bf c}^{\prime}.$ Hence, we have the following equality.

\begin{equation}
\sum_{{{\bf c} \equiv {\bf a}  \, \rm{mod} \,   {\bf f}} \atop
{{\bf l}\, | \, {\bf c}} }
\,\, {\frac{1}{N{\bf c}^s}} \,\, = \,\,
{\frac{1}{N{\bf l}^s}}\sum_{{\bf c}^{\prime}\equiv {\bf a}^{\prime\prime}
\, \rm{mod} \,
{\bf f} } \,\, {\frac{1}{N{\bf c}^{\prime \, s}}} =
{\frac{1}{N{\bf l}^{s}}}\,
\zeta_{{\bf f}} ({\bf a}^{\prime\prime}, s)
\label{2.5}\end{equation}

\noindent Formulas (\ref{2.3}), (\ref{2.4}) and (\ref{2.5}) lead
to the following equality:

\begin{equation}
\zeta_{{\bf f}} ({\bf a}, s) -
{\frac{1}{N{\bf l}^{s}}} \zeta_{{\bf f}} ({\bf l^{-1}} {\bf a}, s)
= \sum_{{{\bf a}^{\prime} \, \rm{mod} \,  {\bf l f} } \atop { {\bf a}^{\prime}
\equiv {\bf a}  \, \rm{mod} \,  {\bf f} }}
\zeta_{{\bf l f}} ({\bf a}^{\prime}, s).
\label{2.51}\end{equation}

For all ${\bf f}$ coprime to ${\bf l}$ and for all ${\bf b}$
coprime to ${\bf l f}$, equality (\ref{2.51}) gives:

\begin{equation}
Res_{K_{{\bf l} {\bf f}}/K_{{\bf f}}} \,\, \Theta_{s} ({\bf b}, {\bf l} {\bf f})
= (1 - ({\bf l}, \, F)^{-1} N{\bf l}^{s}) \,
\Theta_{s} ({\bf b}, {\bf f})
\label{2.52}\end{equation}
\medskip

\noindent
Indeed we easily check that:
\begin{equation}
Res_{K_{{\bf l} {\bf f}}/K_{{\bf f}}} \, (N {\bf b}^{s+1} - ({\bf b}, \, F))
\sum_{{{\bf a}^{\prime} \, \rm{mod} \,  {\bf l f} }}
\zeta_{{\bf l f}} ({\bf a}^{\prime}, - s) ({\bf a}^{\prime},\, F)^{-1} =
\nonumber\end{equation}
\begin{equation}
(N {\bf b}^{s+1} - ({\bf b}, \, F))
\sum_{{{\bf a}  \, \rm{mod} \,  {\bf f} }}
\sum_{{{\bf a}^{\prime} \, \rm{mod} \,  {\bf l f} } \atop  {{\bf a}^{\prime}
\equiv {\bf a}  \, \rm{mod} \,  {\bf f}}}
\zeta_{{\bf l f}} ({\bf a}^{\prime}, - s) ({\bf a},\, F)^{-1} =
\nonumber\end{equation}
\begin{equation}
(N {\bf b}^{s+1} - ({\bf b}, \, F))
\sum_{{{\bf a} \, \rm{mod} \,  {\bf f}}}(\zeta_{{\bf f}} ({\bf a}, - s) -
N{\bf l}^{s} \zeta_{{\bf f}} ({\bf l}^{-1} {\bf a}, - s))
({\bf a},\, F)^{-1} =
\nonumber\end{equation}
\begin{equation}
(N {\bf b}^{s+1} - ({\bf b}, \, F))
(\sum_{{{\bf a} \, \rm{mod} \,  {\bf f}}}\zeta_{{\bf f}} ({\bf a}, - s)
({\bf a},\, F)^{-1}
- ({\bf l}, \, F)^{-1} N{\bf l}^{s} \,
\zeta_{{\bf f}} ({\bf l}^{-1} {\bf a}, - s) ({\bf l}^{-1} {\bf a},\, F)^{-1}) =
\nonumber
\end{equation}
\begin{equation}
(1 - ({\bf l}, \, F)^{-1} N{\bf l}^{s})
(N {\bf b}^{s+1} - ({\bf b}, \, F))
\sum_{{{\bf a} \, \rm{mod} \,  {\bf f}}}\zeta_{{\bf f}} ({\bf a}, - s)({\bf a},\, F)^{-1}
\nonumber\end{equation}

\begin{lemma}\label{Lemma 2.1}
Let ${\bf f}\, | \, {\bf f}^{\prime}$ be ideals of ${\mathcal
O}_K$ coprime to ${\bf b}.$ Then, we have the following.
\begin{equation}
Res_{K_{{\bf f}^{\prime}}/K_{{\bf f}}} \,\, \Theta_{s} ({\bf b}, {\bf f}^{\prime})
= \bigl( \, \prod_{{{\bf l}\nmid {\bf f}} \atop
{{\bf l} \, | \, {\bf f}^{\prime}}} \, (1 - ({\bf l}, \, F)^{-1} N{\bf l}^{s}) \, \bigr) \,\,
\Theta_{s} ({\bf b}, {\bf f})
\label{2.53}\end{equation}
\end{lemma}
\begin{proof}
The lemma follows from (\ref{2.2}) and (\ref{2.52}).
\end{proof}

\begin{remark}
The property of higher Stickelberger elements given by the above Lemma will translate naturally
into the Euler System property of the special elements in Quillen $K$--theory constructed
in \S5 below.
\end{remark}

In what follows, for any given abelian extension $F/K$ of conductor $\bf f$, we
consider the field extensions $F (\mu_{l^k}) / K$, for all $k \geq
0$ and a fixed prime $l$, where $\mu_{l^k}$ denotes the group of roots of unity of order dividing $l^k$. We let ${\bf f}_k$ denote the conductor of
the abelian extension $F (\mu_{l^k}) / K.$ We suppress from the
notation the explicit dependence of ${\bf f}_k$ on $l$, since the
prime $l$ will be chosen and fixed once and for all in this paper.


\noindent
\section{Basic facts about algebraic $K$-theory}

\subsection{The Bockstein sequence and the Bott element}

\noindent Let us fix  a prime number $l$. For a ring $R$ we consider the Quillen $K$-groups
$$K_m
(R) := \pi_{m} (\Omega B Q P (R)) := [S^m, \, \Omega B Q P (R)]$$
(see \cite{Q1}) and the $K$-groups with coefficients
$$K_m (R, \,
\Z/l^k) := \pi_{m} (\Omega B Q P (R), \, \Z/l^k) := [M^{m}_{l^k},
\, \Omega B Q P (R)]$$ defined by Browder and Karoubi in \cite{Br}.
Quillen's $K$--groups  can also be computed using Quillen's plus
construction as $K_n (R) := \pi_{n} (BGL (R)^{+}).$  Any unital
homomorphism of rings $\phi\, :\, R \rightarrow R^{\prime}$
induces natural homomorphisms
$$\phi_{R \mid R^{\prime}}\, :\, K_{m} (R, \, \diamondsuit)
\stackrel{}{\longrightarrow}
K_m (R^{\prime}, \, \diamondsuit)$$
where $K_{m} (R, \, \diamondsuit)$ denotes either $K_{m} (R)$ or
$K_{m} (R, \, \Z/l^k).$
\bigskip

\noindent  Quillen $K$-theory and $K$-theory with coefficients admit
product structures:
$$K_n (R, \, \diamondsuit) \times K_m (R, \, \diamondsuit)
\stackrel{\ast}{\longrightarrow} K_{m + n} (R, \, \diamondsuit)$$
(see \cite{Q1} and \cite{Br}.) These induce graded ring structures on the groups $\bigoplus_{n
\geq 0} K_n (R, \, \diamondsuit).$
\medskip

\noindent  For a topological space $X$, there is a Bockstein exact
sequence
$$ \stackrel{}{\longrightarrow} \pi_{m+1} (X, \, \Z/l^k)
\stackrel{b}{\longrightarrow} \pi_m (X)
\stackrel{l^k}{\longrightarrow} \pi_m (X)
\stackrel{}{\longrightarrow} \pi_m (X, \, \Z/l^k)
\stackrel{}{\longrightarrow} $$ In particular, if we take $X :=
\Omega B Q P (R))$, we get the Bockstein exact sequence in
$K$-theory given by
\begin{equation}\label{Bokstein}\stackrel{}{\longrightarrow} K_{m+1} (R, \, \Z/l^k)
\stackrel{b}{\longrightarrow} K_m (R) \stackrel{l^k}{\longrightarrow}
K_m (R) \stackrel{}{\longrightarrow} K_m (R, \, \Z/l^k)
\stackrel{}{\longrightarrow}\end{equation}
\medskip

\noindent
For any discrete group $G$, we have:
$$
\pi_{n} (BG) \,\, = \,\,
\left\{
\begin{array}{lll}
G&\rm{if}&n = 1\\
0&\rm{if}&n > 1.\\
\end{array}\right.
$$
Consequently, for a commutative group $G$ and $X := BG$ the
Bockstein map $b$ gives an isomorphism $b\, : \, \pi_{2} (BG, \,
\Z/l^k) \stackrel{\cong}{\longrightarrow} G [l^k]$. Here, $G[m]$ denotes the $m$--torsion subgroup of the
commutative group $G$, for all $m\in\Bbb N$.
\medskip

\noindent
For a commutative ring with identity $R$ we have
$GL_1 (R) = R^{\times}.$ Assume that $\mu_{l^k} \subset R^{\times}$. Then
$R^{\times} [l^k] = \mu_{l^k}.$
Let $\beta$ denote the natural composition of maps:
$$\xymatrix{\mu_{l^k}
\ar[r]^{b^{-1}\qquad\quad} &\pi_{2} (BGL_1 (R); \Z/l^k)\ar[r]
&\pi_{2} (BGL (R); \Z/l^k)\ar[d]&\\
&&\pi_{2} (BGL (R)^{+}; \Z/l^k)\ar[r]^{\qquad=}
&K_2 (R, \, \Z/l^k)}$$
We fix a generator $\xi_{l^k}$
of $\mu_{l^k}$. We define the Bott element
\begin{equation}\label{Bott}\beta_{k} :=
\beta (\xi_{l^k}), \qquad  \beta_k\in K_2 (R; \, \Z/l^k)\end{equation}
as the image of $\xi_{l^k}$ via $\beta$.  Further, we let
$$\beta_{k}^{\ast \, n} :=
\beta_{k} \ast \dots \ast \beta_{k} \in K_{2n} (R; \, \Z/l^k).$$
The Bott element $\beta_{k}$ depends of course on the ring
$R$. However, we suppress this dependence from the notation since it will be always
clear where a given Bott element lives. For example, if
$\phi\, :\, R \rightarrow R^{\prime}$ is a
homomorphism of commutative rings containing $\mu_{l^k}$,  then it is clear from
the definitions that the map
$$\phi_{R \mid R^{\prime}}\, :\, K_{2} (R; \, \Z/l^k)
\stackrel{}{\longrightarrow}
K_2 (R^{\prime}, \, \Z/l^k)$$ transports the Bott element for
$R$ into the Bott element for $R^{\prime}$. By a slight abuse of notation, this will be written
as  $\phi_{R \mid R^{\prime}} (\beta_{k}) = \beta_{k}.$
\medskip

Dwyer and Fiedlander \cite{DF} constructed the \' etale $K$-theory $K_\ast^{et}(R)$ and \'etale $K$-theory with coefficients $K_\ast^{et}(R, \Bbb Z/l^k)$ for any commutative, Noetherian $\Bbb Z[1/l]$--algebra $R$.
Moreover, they proved that if $l>2$ then there are natural graded
ring homomorphisms, called the Dwyer-Friedlander maps:
\begin{equation}
K_{\ast} (R) \,\,
{\stackrel{}{\longrightarrow}} \,\,
K_{\ast}^{et} (R)
\label{infinitecoefDFmap}
\end{equation}
\begin{equation}
K_{\ast} (R; \Z/l^k ) \,\,
{\stackrel{}{\longrightarrow}} \,\,
K_{\ast}^{et} (R; \Z/l^k).
\label{finitecoefDFmap}
\end{equation}
If $R$ has finite $\Z/l$-cohomological dimension  then there are Atiyah-Hirzebruch
type spectral sequences (see \cite{DF}, Propositions 5.1, 5.2):
\begin{equation}
E_{2}^{p, -q} = H^p(R ; \Z_l (q/2)) \Rightarrow K_{q - p}^{et} (R).
\label{infinitecoefDFspecseq}
\end{equation}
\begin{equation}
E_{2}^{p, -q} = H^p(R ; \Z/l^k (q/2)) \Rightarrow K_{q - p}^{et} (R; \Z/l^k).
\label{finitecoefDFspecseq}
\end{equation}
\medskip

\noindent Throughout, we will denote by $r_{k'/k}$ the reduction maps at the level of coefficients
$$r_{k'/k}: K_\ast(R; \Bbb Z/l^{k'})\to K_\ast(R; \Bbb Z/l^k),$$
$$r_{k'/k}: K_\ast^{et}(R; \Bbb Z/l^{k'})\to K_\ast^{et}(R; \Bbb Z/l^k),$$
for any $R$ as above and $k'\geq k$.

\subsection{$K$-theory of finite fields}
Let $\F_q$ be the finite field with $q$ elements. In [Q3], Quillen
proved that:
$$
K_n (\F_q) \,\, \simeq \,\,
\left\{
\begin{array}{lll}
\Z &\rm{if}&n = 0\\
0&\rm{if}&n = 2m \quad \text{and} \quad m > 0\\
\Z/(q^m - 1)\Z & \rm{if} & n = 2m -1 \quad \text{and} \quad m > 0\\
\end{array}\right.
$$

\medskip

\noindent
Moreover, in loc.cit, pp. 583-585, it is also showed that for an inclusion
$i \, : \, \F_q \rightarrow \F_{q^f}$ of finite fields and all $n \geq 1$
the natural map
$$i \, : \, K_{2n-1} (\F_q) \rightarrow K_{2n-1} (\F_{q^f})$$
is injective and the transfer map
$$N \, : \, K_{2n-1} (\F_{q^f}) \rightarrow K_{2n-1} (\F_{q})$$
is surjective, where we simply write $i$ instead of
$i_{\F_q \mid \F_{q^f}}$ and $N$ instead of $Tr_{\F_{q^f}/\F_q}.$
Further (see loc.cit.,  pp. 583-585), $i$ induces an isomorphism
$$K_{2n-1} (\F_q) \cong K_{2n-1} (\F_{q^f})^{G(\F_{q^f} / \F_{q})}$$
and the $q$--power Frobenius automorphism $Fr_{q}$
(the canonical generator of $G(\F_{q^f} / \F_{q})$) acts on $K_{2n-1} (\F_{q^f})$
via multiplication by $q^n.$  Observe that
$$i \circ N = \sum_{i = 0}^{f-1} Fr_{q}^{i}.$$
Hence, we have the equalities
$$\text{Ker}\, N =  \, K_{2n-1} (\F_{q^f})^{Fr_{q} - Id}\,  =
 \, K_{2n-1} (\F_{q^f})^{q^n - 1}$$
since $\text{Ker} \, N$ is the kernel of multiplication
by $\sum_{i = 0}^{f-1} \, q^{ni} = {{q^{nf} - 1} \over {q^n - 1}}$
in the cyclic group $K_{2n-1} (\F_{q^f}).$
In particular, this shows that the norm map $N$ induces the following
group isomorphism
$$K_{2n-1} (\F_{q^f})_{G(\F_{q^f} / \F_{q})} \cong K_{2n-1} (\F_q). $$

\medskip

\noindent
By the Bockstein exact sequence \eqref{Bokstein} and Quillen's results above, we observe that
$$K_{2n} (\F_q,\, \Z/l^k) \stackrel{b}{\longrightarrow} K_{2n -1} (\F_q)[l^k]$$
is an isomorphism. Hence, $K_{2n} (\F_q,\, \Z/l^k)$ is a cyclic group.

Let us assume that $\mu_{l^k} \subset \F_{q}^{\times}$ (i.e. $l^k\mid q-1$.)
In this case, Browder \cite{Br} proved that the element
$\beta_{k}^{\ast \, n}$ is a generator of $K_{2n} (\F_q,\, \Z/l^k)$.
Dwyer and Friedlander \cite{DF} proved that there is a natural isomorphism of
graded rings:
$$K_{\ast} (\F_q,\, \Z/l^k) \stackrel{\cong}{\longrightarrow}
K_{\ast}^{et} (\F_q, \, \Z/l^k ).$$
By abuse of notation, let $\beta_{k}$ denote the image of the Bott
element defined in \eqref{Bott} via the natural isomorphism:
$$K_{2} (\F_q,\, \Z/l^k) \stackrel{\cong}{\longrightarrow}
K_{2}^{et} (\F_q, \, \Z/l^k ).$$
Then by Theorem 5.6 in \cite{DF} multiplication with
$\beta_{k}$ induces isomorphisms:
$$ \times \,  \beta_{k}  \, : \, K_{i} (\F_q,\, \Z/l^k)
\stackrel{\cong}{\longrightarrow} K_{i+2} (\F_q, \, \Z/l^k ),$$
$$ \times \, \beta_{k}  \, : \, K_{i}^{et} (\F_q,\, \Z/l^k)
\stackrel{\cong}{\longrightarrow} K_{i+2}^{et} (\F_q, \, \Z/l^k ).$$
In particular,  if  $l^k \, | \, q-1$ and $\alpha$ is a generator of
$K_{1} (\F_q,\, \Z/l^k) = K_{1} (\F_q) /l^k$, then
the element $\alpha \ast \beta_{k}^{\ast \, n-1}$ is a generator of the
cyclic group $K_{2n-1} (\F_q,\, \Z/l^k).$

\subsection{$K$-theory of number fields and their rings of integers}
Let $F$ be a number field. As usual, ${\mathcal O}_{F}$ denotes
the ring of integers in $F$ and $k_v$ is the residue field
for a prime $v$ of ${\mathcal O}_{F}.$
For a finite set of primes
$S$ of ${\mathcal O}_{F}$ the ring of $S$-integers of $F$ is denoted
${\mathcal O}_{F, S}.$
\bigskip

\noindent
Quillen \cite{Q2} proved that $K_{n}({\mathcal O}_{F})$ is a finitely generated
group for every $n \geq 0.$
Borel computed the ranks of the groups $K_{n}({\mathcal O}_{F})$ as follows:
$$
 K_n ({\mathcal O}_{F}) \otimes_{\Z} \Q \,\,\simeq \,\,
\left\{
\begin{array}{lll}
\Q   & \text{if} &  n=0\\
\Q^{r_{1} + r_{2} - 1} & \text{if} & n = 1\\
0 & \text{if} & n = 2m \quad \text{and} \quad n > 0\\
\Q^{r_1 + r_2}   & \text{if} & n \equiv 1 \mod 4 \quad
\text{and} \quad n \not= 1 \\
\Q^{r_2}  & \text{if} & n \equiv 3 \mod 4 \\
\end{array}\right.
$$

\noindent
We have the following localization
exact sequences in Quillen $K$-theory and $K$-theory with coefficients \cite{Q1}.

$$\stackrel{}{\longrightarrow} K_{m} ({\mathcal O}_{F}, \, \diamondsuit)
\stackrel{}{\longrightarrow} K_m (F, \, \diamondsuit)
\stackrel{\partial_F}{\longrightarrow}
\bigoplus_{v} K_{m-1} (k_v, \, \diamondsuit)
\stackrel{}{\longrightarrow} K_{m-1} ({\mathcal O}_{F}, \, \diamondsuit)
\stackrel{}{\longrightarrow}$$
\medskip

\noindent
Let $E/F$ be a finite extension. The natural maps in $K$-theory
induced by the embedding $i\, :\, F \rightarrow E$ and $\sigma\, :\,
E \rightarrow E,$ for $\sigma \in G(E/F)$, will be denoted
for simplicity by $i \, :\, K_{m} (F, \, \diamondsuit) \stackrel{}{\longrightarrow}
K_m (E, \, \diamondsuit)$ and $\sigma \, : \,
K_{m} (E, \, \diamondsuit) \stackrel{}{\longrightarrow}
K_m (E, \, \diamondsuit).$ Observe that  $i := i_{F \mid  E}$ and
$\sigma := \sigma_{E \mid E}$, according to the notation in
section 3.1.
\medskip

\noindent
In addition to the natural maps
$i,$ $\sigma,$ $\partial_F,$ $\partial_E,$ and product structures $\ast$ for
$K$-theory of $F$ and $E$ introduced above, we have (see \cite{Q1}) the transfer map

$$Tr_{E/F}\, :\, K_{m} (E, \, \diamondsuit) \stackrel{}{\longrightarrow}
K_m (F, \, \diamondsuit)$$
and the reduction map
$$r_v\, :\, K_{m} ({\mathcal O}_{F, S}, \, \diamondsuit)
\stackrel{}{\longrightarrow} K_m (k_v, \, \diamondsuit)$$
for any prime $v \notin S.$
\medskip

\noindent
The maps discussed above enjoy many compatibility properties.
For example,  $\sigma$ is naturally compatible with $i,$
$\partial_F,$ $\partial_E$, the product structure $\ast,$
$Tr_{E/F}$ and $r_w$ and $r_v.$ See e.g. \cite{Ba1} for explanations
of some of these compatibility properties.
Let us mention below two nontrivial such compatibility properties which will be used in what follows.
By a result of Gillet \cite{Gi}, we have the
following commutative diagrams in Quillen $K$-theory and $K$-theory with
coefficients:
\begin{equation}\label{Gillet}\xymatrix{
K_{m} (F, \, \diamondsuit) \times  K_{n} ({\mathcal O}_{F}, \, \diamondsuit)
\ar@<0.1ex>[d]^{\partial_F \times id}  \ar[r]^{\ast} &  K_{m+n} (F, \,
\diamondsuit)
\ar@<0.1ex>[d]^{\partial_F}\\
\bigoplus_{v} \, K_{m-1} (k_v, \, \diamondsuit) \times
K_{n} ({\mathcal O}_{F}, \, \diamondsuit) \quad\quad
\ar[r]^{\qquad\ast} & \quad\quad \bigoplus_{v} \, K_{m + n -1} (k_v, \,
\diamondsuit)
}\end{equation}
\medskip

\noindent
Let $E/F$ be a finite extension unramified over a prime
$v$ of ${\mathcal O}_{F}.$  Let $w$ be a prime of
${\mathcal O}_{E}$ over $v.$ From now on, we will write
$N_{w/v} \, := \, Tr_{k_w / k_v}.$
The following diagram shows the compatibility of transfer
with the boundary map in localization sequences for Quillen $K$-theory
and $K$-theory with coefficients.

\begin{equation}\label{transfer-compatibility}\xymatrix{
K_{m} (E, \, \diamondsuit)\ar[d]_{Tr_{E/F}}
\ar[r]^{\partial_E\qquad\qquad} &
\quad\qquad \bigoplus_{v} \bigoplus_{w\, | \,v} K_{m-1} (k_w, \, \diamondsuit)
\ar[d]^{\bigoplus_{v} \bigoplus_{w\, | \,v} N_{w/v}}\\
K_{m} (F, \, \diamondsuit)
\ar[r]^{\partial_F} &
\bigoplus_{v} \, K_{m -1} (k_v, \, \diamondsuit)}\end{equation}
\medskip
\noindent
where the direct sums are taken with respect primes $v$ in $F$ and $w$ in $E$, respectively.


\noindent
\section{Construction of $\Lambda$ and $\Lambda^{et}$ and first applications}

In this section, we construct special elements
in $K$-theory and \' etale $K$-theory with coefficients, under
the assumption that for some fixed $m > 0$ the Stickelberger elements
$\Theta_{m} ({\bf b}, {\bf f}_k)$ annihilate
$K_{2m} ({\mathcal O}_{F_{l^k}})$  for all $k \geq 0.$ This will produce
special elements in $K$-theory and \' etale $K$-theory without coefficients which are of primary
importance in our construction of Euler systems given in \S5. These constructions will also give us the
Stickelberger splitting maps $\Lambda := \Lambda_{n}$
and $\Lambda^{et} := \Lambda_{n}^{et}$ announced in the introduction. As a byproduct, we obtain a direct proof of the
annihilation of the groups divisible elements $div(K_{2n}(F)_l)$, for all $n>0$, generalizing the results of \cite{Ba1}.

All the results in this section are stated for both $K$-theory and
\' etale $K$-theory. However, detailed proofs will be given only
in the case of $K$-theory since the proofs in the case of \' etale
$K$-theory are very similar. The key idea in transferring the
$K$--theoretic constructions to \'etale $K$--theory is the
following. Replace Gillet's result \cite{Gi} for $K$-theory
(commutative diagram \eqref{Gillet} of \S3) with the compatibility
of the Dwyer-Friedlander spectral sequence with the product
structure (\cite{DF}, Proposition 5.4) combined with Soul{\' e}'s
observation (see \cite{So1}, p. 275) that the localization
sequence in \' etale cohomology (see \cite{So1}, p. 268) is
compatible with the product by the \' etale cohomology of
${\mathcal O}_{F, S}.$

\subsection{Constructing special elements in $K$--theory with coefficients} Let $L$ be a number field, such that $\mu_{l^k} \subset L.$ Let $S$ be a finite set
of prime ideals of ${\mathcal O}_{L}$ containing all primes over $l.$
Let $i \in \N$ and let $m \in \Z$, such that $i+2m>0$.
Then, for $R = L$ or $R = {\mathcal O}_{L, S}$ there is a natural group isomorphism
(see \cite{DF} Theorem 5.6):
\begin{equation}
K_{i}^{et} (R; \Z/l^k)
\, {\stackrel{\cong}{\longrightarrow}} \,
K_{i + 2 m}^{et} (R; \Z/l^k)
\label{etalebottmult}\end{equation}
which sends $\eta$ to
$\eta \ast \beta_{k}^{\ast \,  m}$ for any
$\eta \in K_{i}^{et} (R; \Z/l^k).$
If $m \geq 0$ this isomorphism is just the multiplication by
$\beta_{k}^{\ast \, m}.$
If $m < 0$ and $i + 2m > 0$, then the isomorphism (\ref{etalebottmult}) is
the inverse of the multiplication by $\beta_{k}^{\ast \, - m}$ isomorphism:
\begin{equation}
\ast \,\, \beta_{k}^{\ast \, - m}\,  : \,
K_{i + 2m}^{et}  (R; \Z/l^k) \,\, {\stackrel{\cong}{\longrightarrow}} \,\,
K_{i}^{et} (R; \Z/l^k).
\label{etalebottmult1}\end{equation}
\medskip

\noindent
Now, let us consider Quillen $K$-theory.
If $m \geq 0$, there is a natural homomorphism
\begin{equation}
\ast \, \, \beta^{\ast \, m} \, : \: K_{i}  (R; \Z/l^k)
\rightarrow K_{i + 2 m} (R; \Z/l^k)
\label{ktheorybottmult0}\end{equation}
which is just multiplication by  $\beta_{k}^{\ast \, m}.$
The homomorphism (\ref{ktheorybottmult0}) is compatible with the isomorphism
(\ref{etalebottmult}) via the Dwyer-Friedlander map.
If $m < 0$ and $i + 2m > 0$,  then take the homomorphism
\begin{equation}
t(m) \, : \: K_{i}  (R; \Z/l^k)
\rightarrow K_{i + 2 m} (R; \Z/l^k)
\label{ktheorybottmult00}\end{equation}
to be the unique homomorphism
which makes the following diagram commutative.
$$\xymatrix{
K_{i}  (R; \Z/l^k)
\ar@<0.1ex>[d]^{}\,  \quad \ar[r]^{\quad t (m)\quad }  &  \quad K_{i + 2 m}  (R; \Z/l^k)   \\
K_{i }^{et}  (R; \Z/l^k) \quad
\ar[r]^{\qquad (\ast \, \beta_{k}^{\ast \, - m})^{-1}\qquad }
&  \quad K_{i + 2 m}^{et}  (R; \Z/l^k)  \ar@<0.1ex>[u]^{} }
\label{ktheorybottmult}$$
The left vertical arrow is the Dwyer-Friedlander map, while the right vertical arrow is
the Dwyer-Friedlander splitting map (see \cite{DF}, Proposition 8.4.)
The latter map is obtained as the multiplication of the
inverse of the isomorphism
$K_{i^{\prime}}  (R; \Z/l^k) {\stackrel{\cong}{\longrightarrow}}
K_{i^{\prime}}^{et}  (R; \Z/l^k),$
for $i^{\prime} = 1 $ or $i^{\prime} = 2,$ by a nonnegative power of the Bott element
$\beta_{k}^{\ast \, m^{\prime}},$ with $m^{\prime} \geq 0$
(see the proof of Proposition 8.4 in \cite{DF}.)

\begin{remark} \label{Remark 7.1 on coeff change oF DF splitting}
It is clear that the Dwyer-Friedlander splitting
from \cite{DF}, Proposition 8.4 is compatible with the maps
$\Z/l^{j} \rightarrow \Z/l^{j-1}$ at the level of coefficients, for all $1 \leq j \leq k.$
Consequently, the map $t(m)$ is naturally compatible with these maps.
In addition, $t(m)$ is naturally compatible with the ring embedding
$R \rightarrow R^{\prime}$, where $R^{\prime} = L^{\prime}$ or
$R^{\prime} = {\mathcal O}_{L^{\prime}, S}$ for a number field extension
$L^{\prime}/L.$ Let
$$t^{et} (m) := (\ast \, \beta_{k}^{\ast \, - m})^{-1}.$$
It is clear from the above diagram that $t(m)$ and $t^{et} (m)$ are
naturally compatible with the Dwyer-Friedlander maps.
\end{remark}

\begin{lemma}\label{Lemma 7.2}
Let $L = F (\mu_{l^k})$ and let $i > 0$ and $m < 0$, such that $i + 2 m > 0.$
Then, for $R = L$ or $R = {\mathcal O}_{L, S}$, the natural group homomorphisms
$t^{et} (m)$  and $t(m)$ have the following properties:
\begin{equation}
t^{et} (m) (\alpha)^{\sigma_{{\bf a}}} =
t^{et} (m) (\alpha^{N  {\bf a}^m \sigma_{{\bf a}}})
\label{ktheorybottmult1}\end{equation}
\begin{equation}
t(m) (\alpha)^{\sigma_{{\bf a}}} =
t(m) (\alpha^{N  {\bf a}^m \sigma_{{\bf a}}})
\label{ktheorybottmult1}\end{equation}
for any ideal ${\bf a}$ of $O_F$ coprime to ${\bf f}_k$ and for $\alpha \in K_{i}^{et} (R; \Z/l^k)$ and
$\alpha \in K_{i} (R; \Z/l^k)$, respectively.
\end{lemma}

\begin{lemma}\label{Lemma 7.3}
If $i \in \{1, 2\},$ $\alpha \in K_{i} (R; \Z/l^k)$ and $n + m > 0$
then
\begin{equation}
t^{et} (m) (\alpha \ast \beta_{k}^{\ast \, n }) =
\alpha \ast \beta_{k}^{\ast \, n + m }.
\label{etalenegativebottmult2}\end{equation}
\begin{equation}
t (m) (\alpha \ast \beta_{k}^{\ast \, n }) =
\alpha \ast \beta_{k}^{\ast \, n + m }.
\label{ktheorynegativebottmult2}\end{equation}
\end{lemma}
\begin{proof}[Proof of Lemmas 4.2 and 4.3] The properties in Lemmas \ref{Lemma 7.2} and \ref{Lemma 7.3}
follow directly from the definition of the
maps $t^{et} (m)$ and $t(m).$
\end{proof}

If $v$ is a prime of ${\mathcal O}_{L, S},$
$m < 0$ and $i + 2m > 0$, then we construct the morphism
\begin{equation}
t_v (m) \, : \: K_{i}  (k_v ; \Z/l^k)
\rightarrow K_{i + 2 m} (k_v ; \Z/l^k)
\label{ktheorybottmult00}\end{equation}
in the same way as we have done for ${\mathcal O}_{L, S}$ or $L.$
Namely, $t_v (m)$ is the homomorphism which makes the following diagram commute.
$$\xymatrix{
K_{i}  (k_v ; \Z/l^k)
\ar@<0.1ex>[d]^{\cong}\,  \quad \ar[r]^{t_v (m)}  &  \quad K_{i + 2 m}  (k_v ; \Z/l^k)   \\
K_{i}^{et}  (k_v; \Z/l^k) \quad
\ar[r]^{(\ast \, \beta_{k}^{\ast \, - m})^{-1}}
&  \quad K_{i + 2 m}^{et}  (k_v; \Z/l^k)  \ar@<0.1ex>[u]^{\cong} }
\label{ktheorybottmult3}$$
The right vertical arrow is
the inverse of the Dwyer-Friedlander map which, in the case of a finite field,
is clearly seen to be equal to the Dwyer-Friedlander splitting map described above.

Similarly to $t^{et} (m)$ we can construct $t_{v}^{et} (m):=(\ast \, \beta_{k}^{\ast \, - m})^{-1}.$
We observe that the maps $t (m)$ and $t_v (m)$ are compatible with the reduction
maps and the boundary maps. In other words, we have the following commutative diagrams.

$$\xymatrix{
K_{i}  ({\mathcal O}_{L, S} ; \, \Z/l^k)
\ar@<0.1ex>[d]^{t (m)}\,  \quad \ar[r]^{r_v}  &  \quad K_{i }  (k_v; \, \Z/l^k)
 \ar@<0.1ex>[d]^{t_v (m)}  \\
K_{i + 2 m}  ({\mathcal O}_{L, S} ; \, \Z/l^k) \quad
\ar[r]^{r_v}
&  \quad K_{i + 2 m}  (k_v; \, \Z/l^k)}
\label{ktheorybottmult4}$$
\medskip

\noindent
$$\xymatrix{
 K_{i} ({\mathcal O}_{L, S}, \, \Z/l^k)
\ar@{>}[d]^{t (m)} \ar[r]^{\partial}  & \bigoplus_{v \in S} \, \,
K_{i-1} (k_v; \, \Z/l^k) \ar@<0.1ex>[d]^{t_v (m)} \\
K_{i + 2m} ({\mathcal O}_{L, S}; \, \Z/l^k) \ar[r]^{\partial} \,\, & \bigoplus_{v
\in S} \, \, K_{i - 1 + 2m} (k_v; \, \Z/l^k) }
\label{CommDiagrQuilEtale1}$$
\medskip
\noindent
Let us point out that we have similar commutative diagrams
for \'etale $K$-theory and the maps $t^{et} (m)$ and $t_{v}^{et} (m).$

\medskip
As observed above, the map $t(m)$ for $m < 0$ has
the same properties as the multiplication by $\beta^{\ast \, m}$
for $m \geq 0.$ So, we make the following.

\begin{definition} For $m < 0$, we define the symbols $$\alpha \ast
\beta^{\ast \, m} : = t (m) (\alpha), \qquad
\alpha_v \ast
\beta^{\ast \, m} : = t_v (m) (\alpha_v),$$
for all
$\alpha \in K_{i} ({\mathcal O}_{L}; \, \Z/l^k)$
and $\alpha_v \in K_{i} (k_v; \Z/l^k))$, respectively.
For $m \geq 0$, the symbols $\alpha \ast \beta^{\ast \, m}$
and $\alpha_v \ast \beta^{\ast \, m}$ denote the usual products.
\end{definition}
\medskip

Let $m  > 0$ be a natural number. Throughout the rest of this section
we assume that $\Theta_{m} ({\bf b}, {\bf f}_k)$ annihilates
$K_{2m} ({\mathcal O}_{F_{l^k}})$ for all $k \geq 0.$
For a prime $v$ of ${\mathcal O}_{F}$, let
$k_v$ be its residue field and $q_v$ the cardinality of $k_v$ . Similarly, for any prime
$w$ of ${\mathcal O}_{F_{l^k}}$, we let $k_w$ be its residue field.
We put $E := F_{l^k}.$
If $v \not\,\mid l,$ we observe that $k_w = k_v (\xi_{l^k})$, since the corresponding
local field extension $E_{w}/F_{v}$ is unramified.
For any finite set $S$ of primes in $O_F$ and any $k\geq 0$,
there is an exact sequence (see \cite{Q2}):

$$
0 \stackrel{}{\longrightarrow} K_{2 m} ({\mathcal O}_{F_{l^k}})
\stackrel{}{\longrightarrow} K_{2 m} ({\mathcal O}_{F_{l^k}, \, S})
\stackrel{\partial}{\longrightarrow} \bigoplus_{v \in S} \bigoplus_{w | v}
K_{2 m -1} (k_w) \stackrel{}{\longrightarrow} 0
$$

Let $\xi_{w, k} \in K_{2 m - 1} (k_w)_l$ be a generator of the $l$-torsion part of
$K_{2 m -1} (k_w)$.
Pick an element $x_{w, k} \in K_{2 m} ({\mathcal O}_{F_{l^k}, S})_l$ such that
${\partial} (x_{w, k}) = \xi_{w, k}.$ Obviously,
$x_{w, k}^{\Theta_{m} ({\bf b}, {\bf f}_k)}$ does not depend on the choice of $x_{w, k}$
since $\Theta_{m} ({\bf b}, {\bf f}_k)$ annihilates
$K_{2 m} ({\mathcal O}_{F_{l^k}}).$ If $\text{ord}(\xi_{w, k}) = l^a$, then
$x_{w, k}^{l^a} \in K_{2 m} ({\mathcal O}_{F_{l^k}}).$ Hence,
$(x_{w, k}^{\Theta_{m} ({\bf b}, {\bf f}_k)})^{l^a} =
(x_{w, k}^{l^a})^{\Theta_{m} ({\bf b}, {\bf f}_k)} = 0.$
Consequently, there is a well defined map:

$$\Lambda_{m}\, : \, \bigoplus_{v \in S} \bigoplus_{w | v} K_{2 m - 1} (k_w)_l
\stackrel{}{\longrightarrow} K_{2 m} ({\mathcal O}_{F_{l^k}, \, S})_l ,$$
\begin{equation}
\Lambda_{m} (\xi_{w, k}) \, := \, x_{w, k}^{\Theta_{m} ({\bf b}, {\bf f}_k)}.
\label{firststepsplitting}\end{equation}
\medskip

If $R$ is either a number field $L$ or its ring of
$(S,l)$--integers ${\mathcal O}_{L, S}[1/l]$, for some finite set $S\subseteq {\rm Spec}(\mathcal O_L)$,  Tate proved in
\cite{Ta2} that there is a natural isomorphism:
$$K_{2} (R)_l \,
{\stackrel{\cong}{\longrightarrow}} \, K_{2}^{et} (R).$$
Dwyer and
Friedlander \cite{DF} proved that the natural maps:
$$K_{j} (R; \Z/l^k ) \,
{\stackrel{}{\longrightarrow}} \, K_{j}^{et} (R; \Z/l^k),$$
are surjections for $j\geq 1$ and isomorphisms for $j = 1, 2.$
As explained in \cite{Ba2}, for any number
field $L$, any finite set $S \subset \text{Spec} ({\mathcal
O}_{L})$ and any $j\geq 1$, we have the following commutative diagrams with exact
rows and (surjective) Dwyer-Friedlander maps as vertical arrows.
$$\xymatrix{
0  \ar[r]^{} & K_{2j} ({\mathcal O}_{L})_l
\ar@{>>}[d]^{} \ar[r]^{}  & K_{2j} ({\mathcal O}_{L, S})_l
\ar@{>>}[d]^{} \ar[r]^{\partial}  & \bigoplus_{v \in S} \, \,
K_{2j-1} (k_v)_l \ar@<0.1ex>[d]^{\cong} \ar[r]^{} & 0 \\
0  \ar[r]^{} & K_{2j}^{et} ({\mathcal O}_{L}[1/l])  \ar[r]^{}  &
K_{2j}^{et} ({\mathcal O}_{L, S}[1/l]) \ar[r]^{\partial^{et}}  & \bigoplus_{v
\in S} \, \, K_{2j-1}^{et} (k_v) \ar[r]^{} & 0}
\label{CommDiagrQuilEtale1}$$ For $j = 1$, the left and the middle
vertical arrows in the above diagram are also isomorphisms,
according to Tate's theorem. If the Quillen-Lichtenbaum conjecture holds, then
these are isomorphisms for all $j>0$.
\medskip

Our assumption that $\Theta_{m} ({\bf b}, {\bf
f}_k)$ annihilates $K_{2 m} ({\mathcal O}_{F_{l^k}})$ for all $k
\geq 0$ implies that $\Theta_{m} ({\bf b}, {\bf f}_k)$ annihilates
$K_{2 m}^{et} ({\mathcal O}_{F_{l^k}}[1/l])$, for all $k \geq 0.$
In the diagram above, let $y_{w, k}$ and
$\zeta_{w, k}$ denote the images of  $x_{w, k}$ and $\xi_{w, k}$
via the middle vertical and right vertical arrows, respectively.
Then, we define
$$\Lambda_{m}^{et} (\zeta_{w, k}) \, := \, y_{w, k}^{\Theta_{m} ({\bf b}, {\bf f}_k)}.$$
Clearly, the following diagram is commutative
$$\xymatrix{
K_{2 m} ({\mathcal O}_{F_{l^k}, S})_l
\ar@<0.1ex>[d]^{}\,   &  \ar[l]_{\Lambda_m\quad } \bigoplus_{v \in S} \bigoplus_{w | v}
\, K_{2 m - 1} (k_w)_l \ar@<0.1ex>[d]^{\cong} \\
K_{2 m}^{et} ({\mathcal O}_{F_{l^k}, S}[1/l]) &
\ar[l]_{\Lambda_{m}^{et}\quad }  \bigoplus_{v \in S} \bigoplus_{w | v}
\, K_{2 m - 1}^{et} (k_w)_l} \label{CommDiagrQuilEtale3}$$
where the vertical maps are the Dwyer-Friedlander maps.
\begin{lemma}\label{Lemma 7.4}
The maps $\Lambda_{m}$ and $\Lambda_{m}^{et}$ satisfy
the following properties
$$\partial \Lambda_{m} (\xi_{w, k}) \, :=
\, \xi_{w, k}^{\Theta_{m} ({\bf b}, {\bf f}_k)},\qquad
\partial^{et} \Lambda_{m}^{et} (\zeta_{w, k}) \, :=
\, \zeta_{w, k}^{\Theta_{m} ({\bf b}, {\bf f}_k)}.$$
\end{lemma}
\begin{proof}
The lemma follows immediately by compatibility of $\partial$ and
$\partial^{et}$ with the $G (E/F)$ action.
\end{proof}

Let us fix an $n\in\Bbb N$. Let $v$ be a prime in ${\mathcal O}_{F}$ sitting above $p \not= l$
in $\Z$. Let $S := S_v$ be the finite set primes of ${\mathcal
O}_{F}$ consisting of all the primes over $p$ and all the primes over $l.$
Let $k (v)$ be the natural number for which $l^{k(v)}\,
|| \, q_{v}^{n} - 1.$ Observe that if $l \, | \, q_v -1$ then
$k(v) = v_{l}(q_v -1) + v_{l}(n)$ (see e.g. \cite[p. 336]{Ba1}.)

\begin{definition}\label{gamma-l} As in loc.cit. p. 335, let us
define:
\begin{equation}
\gamma_{l} := \prod_{{{\bf l} \, \not \, | \, {\bf f}} \atop
{{\bf l} \, | \, l}} \, (1 - ({\bf l}, \, F)^{-1} N{\bf l}^{n})^{-1} =
\prod_{{{\bf l} \, \not \, | \, {\bf f}} \atop
{{\bf l} \, | \, l}} \, (1 + ({\bf l}, \, F)^{-1} N{\bf l}^{n} +
({\bf l}, \, F)^{-2} N{\bf l}^{2n} + \cdots ).
\nonumber
\end{equation}
If ${\bf l} | \, {\bf f}$
for every ${\bf l} \, | \, l$ then naturally we let $\gamma_l := 1.$
Observe that $\gamma_l$ is a well defined operator on any
$\Z_l [G(F/K)]$-module which is a torsion abelian group with a finite exponent.
\end{definition}

\begin{definition}\label{special-coefficients}
For all $k \geq 0$ and $E := F(\mu_{l^k})$, let us define elements:
$${{\blambda}}_{v, l^k} :=
Tr_{E/F} ( \Lambda_m(\xi_{w, k})\ast
\beta_{k}^{\ast \, n-m})^{N {\bf b}^{n-m} \gamma_l} \in K_{2n} ({\mathcal
O}_{F, S};\, \Z/l^k)).$$
Similarly, define elements:
$$\blambda^{et}_{v, l^k} :=
Tr_{E/F} ( \Lambda_m^{et}(\zeta_{w, k}) \ast
\beta_{k}^{\ast \, n-m})^{N {\bf b}^{n-m} \gamma_l} \in K_{2n}^{et}
({\mathcal O}_{F, S};\, \Z/l^k)).$$
\end{definition}
\medskip

\noindent
Obviously, $\blambda^{et}_{v, l^k}$ is the image of $\blambda_{v, l^k}$ via the Dwyer-Friedlander map.
\bigskip

Let us fix a prime sitting above $v$ in each of the fields $F(\mu_{l^k})$,
such that if $k\leq k'$ and $w$ and $w'$ are the fixed primes in $E =F(\mu_{l^k})$ and
$E':=F(\mu_{l^{k'}})$, respectively, then $w'$ sits above $w$. By the surjectivity of
the transfer maps for $K$-theory of finite fields (see the end of
\S3), we can associate to each $k$ and the chosen prime $w$ in
$E=F(\mu_{l^k})$ a generator $\xi_{w, k}$ of $K_{2m-1}(k_w)_l$
and a generator $\zeta_{w, k}$ of $K_{2m-1}^{et}(k_w)_l$, such that
$$N_{w^{\prime}/w}(\xi_{w^{\prime}, k^{\prime}})= \xi_{w, k},\qquad
N_{w^{\prime}/w}(\zeta_{w^{\prime}, k^{\prime}})= \zeta_{w, k},$$
for all $k\leq k'$, where $w$ and $w'$ are the fixed primes in $E =F(\mu_{l^k})$ and
$E' =F(\mu_{l^{k'}})$, respectively.
\bigskip

\begin{lemma}\label{Lemma 7.5} With notations as above, for every
$k \leq k^{\prime}$ we have
$$r_{k^{\prime}/k} ( N_{w^{\prime}/v}
(\xi_{w^{\prime}, k^{\prime}} \ast \beta_{k^{\prime}}^{\ast \, n-m})) =
N_{w/v} (\xi_{w, k} \ast \beta_{k}^{\ast \, n-m}),$$
$$r_{k^{\prime}/k} ( N_{w^{\prime}/v}
(\zeta_{w^{\prime}, k^{\prime}} \ast \beta_{k^{\prime}}^{\ast \, n-m})) =
N_{w/v} (\zeta_{w, k} \ast \beta_{k}^{\ast \, n-m}).$$
\end{lemma}

\begin{proof} First, let us consider the case $n - m \geq 0.$
The formula follows by the compatibility of the elements
$(\xi_{w, k})_{w}$ with respect to the norm maps, by
the compatibility of Bott elements with respect to the coefficient reduction map
$r_{k^{\prime}/k} (\beta_{k^{\prime}}) = \beta_{k},$
and by the projection formula. More precisely,
we have the following equalities:
\begin{eqnarray}
\nonumber  r_{k^{\prime}/k} ( N_{w^{\prime}/v}(
\xi_{w^{\prime}, k^{\prime}} \, \ast \, \beta_{k^{\prime}}^{\ast \, n-m})) &=& N_{w^{\prime}/v} (r_{k^{\prime}/k} (
\xi_{w^{\prime}, k^{\prime}} \, \ast \, \beta_{k^{\prime}}^{\ast \, n-m})) =\\
\nonumber  N_{w^{\prime}/v} (
\xi_{w^{\prime}, k^{\prime}} \, \ast \, \beta_{k}^{\ast \, n-m})) &=& N_{w/v} (N_{w^{\prime}/w} (\xi_{w^{\prime}, k^{\prime}}) \ast
\beta_{k}^{\ast \, n-m}))= \\
\nonumber &=&  N_{w/v} (\xi_{w, k} \ast \beta_{k}^{\ast \, n-m}).
\end{eqnarray}
\noindent
Next, let us consider the case $n - m < 0.$ Observe that the Dwyer-Friedlander maps commute with
$N_{w/v}$ and $N_{w^{\prime}/v}.$ Hence we can argue in the same way as in the case
$n - m \geq  0$ by using the projection formula for the negative twist
in \'etale cohomology, since for any finite field $\F_q$ with $l \not\,\, \mid   q,$
we have natural isomorphisms coming from the Dwyer-Friedlander
spectral sequence (cf. the end of \S3.1):
\begin{equation}
K_{2j-1}^{et} (\F_q) \cong H^1 (\F_q; \, \Z_{l} (j))
\label{finitefieldDF1}\end{equation}
\begin{equation}
K_{2j - 1}^{et} (\F_q; \, \Z/l^k) \cong H^1 (\F_q; \, \Z/l^k (j)).
\label{finitefieldDF2}\end{equation}
\end{proof}

\begin{lemma}\label{Lemma 7.6}
For all $0\leq k \leq k^{\prime}$, we have
$$r_{k^{\prime}/k} (\blambda_{v, l^{k'}}) =
\blambda_{v, l^{k}}$$
$$r_{k^{\prime}/k} (\blambda^{et}_{v, l^{k'}}) = \blambda_{v, l^{k}}.$$
\end{lemma}

\begin{proof}
Consider the following commutative diagram:
$$\xymatrix{
K_{2m} ({\mathcal O}_{E^{\prime}, S}) \quad \ar@<0.1ex>[d]^{Tr_{E^{\prime}/E}}
\ar[r]^{\bigoplus_{w^{\prime} \in S} \partial_{w^{\prime}}} \quad &
\quad
\bigoplus_{w^{\prime} \in S}  K_{2m-1} (k_{w^{\prime}})
\ar@<0.1ex>[d]^{\bigoplus_{w \in S} \bigoplus_{w^{\prime} | w} N_{{w^{\prime} / w}}}\\
K_{2m} ({\mathcal O}_{E, S}) \quad
\ar[r]^{\bigoplus_{w \in S} \partial_{w}} &
\quad \bigoplus_{w \in S} \, K_{2m-1} (k_w)}
\label{diagram 2.5}$$
It follows that we have
$Tr_{E^{\prime}/E} ({x}_{w^{\prime}, k^{\prime}})^{\Theta_{m} ({\bf b}, {\bf f}_k)} =
{x}_{w, k}^{\Theta_{m} ({\bf b}, {\bf f}_k)}.$
Hence the case $n - m \geq 0$ follows by the projection formula:
$$r_{k^{\prime}/k} (Tr_{E^{\prime}/F}
(x_{w^{\prime}, k^{\prime}}^{\Theta_{m} ({\bf b}, {\bf f}_{k^{\prime}})}
\ast \, \beta_{k^{\prime}}^{\ast \, n-m})^{N {\bf b}^{n-m} \gamma_{l}}) =$$
$$ = Tr_{E/F} (Tr_{E^{\prime}/E}
(x_{w^{\prime}, k^{\prime}}^{\Theta_{m} ({\bf b}, {\bf f}_{k^{\prime}})}
\ast \, \beta_{k}^{\ast \, n-m}))^{N {\bf b}^{n-m} \gamma_{l}} =$$
$$ = Tr_{E/F} (x_{w, k}^{\Theta_{m} ({\bf b}, {\bf f}_{k})}
\ast \, \beta_{k}^{\ast \, n-m})^{N {\bf b}^{n-m} \gamma_{l}}.$$
Now, consider the case $n - m < 0.$ We observe that $Tr_{E^{\prime}/E}$ commutes with the
Dwyer-Friedlander map.
Hence $Tr_{E^{\prime}/E}$ also commutes with the splitting of the Dwyer-Friedlander map
since the splitting is a monomorphism. By the Dwyer-Friedlander spectral sequence for any
number field $L$ and any finite set $S$ of prime ideals of ${\mathcal O}_{L}$
containing all primes over $l,$ we have the following isomorphism
\begin{equation}
K_{2j}^{et} ({\mathcal O}_{L, S}) \cong H^2 ({\mathcal
O}_{L, S}; \, \Z_{l} (j + 1))
\label{ringofintegersDF1}\end{equation}
and the following exact sequence
\begin{equation}
0 \rightarrow H^2 ({\mathcal O}_{L, S}; \, \Z/l^k(j + 1))
\rightarrow K_{2j}^{et} ({\mathcal O}_{L, S};\, \Z/l^k) \rightarrow
H^0 ({\mathcal O}_{L, S}; \, \Z/l^k (j)) \rightarrow 0.
\label{finitefieldDF2}\end{equation}
Since ${x}_{w, k}^{\Theta_{m} ({\bf b}, {\bf f}_k)} \in
K_{2m} ({\mathcal O}_{F_{k}, S}),$ its image in
$K_{2m}^{et} ({\mathcal O}_{F_k, S};\, \Z/l^k)$ lies in fact in
in the \'etelae cohomology group $H^2 ({\mathcal O}_{F_k, S}; \, \Z/l^k(m + 1)).$
Hence, one can settle the case $n-m<0$ as well by using the projection formula
for the \'etale cohomology with negative twists.
\end{proof}

\begin{theorem}\label{Theorem 7.7}
For every $k \geq 0$, we have
$$\partial_{F} (\blambda_{v, l^k}) =
N(\xi_{w, k} \ast \beta_{k}^{\ast \, n-m})^{\Theta_{m} ({\bf b}, {\bf f})} \,\, ,$$
$$\partial_{F}^{et} (\blambda^{et}_{v, l^k}) =
N(\zeta_{w, k} \ast \beta_{k}^{\ast \, n-m})^{\Theta_{m} ({\bf b}, {\bf f})} \,\, .$$

\end{theorem}

\begin{proof}
The proof is similar to the proofs of
Theorem 1, pp. 336-340 of \cite{Ba1} and Proposition 2, pp. 221-222 of \cite{BG1}.
The diagram at the end of \S3 gives the following
commutative  diagram of $K$--groups with coefficients

$$\xymatrix{
K_{2n} ({\mathcal O}_{E, S} ;\, \Z/l^k ) \quad
\ar@<0.1ex>[d]^{Tr_{E/F}} \ar[r]^{\partial_{E}\quad } \quad &
\quad\
\bigoplus_{v \in S} \bigoplus_{w\, | \,v} K_{2n-1} (k_w ; \, \Z/l^k )
\ar@<0.1ex>[d]^{N}\\
K_{2n} ({\mathcal O}_{F, S}  ;\, \Z/l^k) \quad
\ar[r]^{\partial_{F}\quad } &
\quad \bigoplus_{v \in S} \, K_{2n-1} (k_v, ;\, \Z/l^k )}\,,
\label{diagram 2.6}$$
where $N := \bigoplus_{v} \bigoplus_{w\, | \,v} N_{w/v}.$
Hence we have $\partial_{F} \circ Tr_{E/F} = N \circ \partial_{E}.$
The compatibilities of some of the natural maps mentioned in \S3
which will be used next can be expressed via the following commutative diagrams,
explaining the action  of the groups $G(E/K)$ and $G(F/K)$ on the $K$--groups with coefficients
in the diagram above.
For $j > 0$ we use the following commutative diagram.
$$\xymatrix{
K_{2j} ({\mathcal O}_{E, S}; \, \Z/l^k ) \quad \ar@<0.1ex>[d]^{\sigma_{{\bf a}}^{-1}}
\ar[r]^{r_w} & \quad
K_{2j} (k_{w}; \, \Z/l^k)
\ar@<0.1ex>[d]^{\sigma_{{\bf a}}^{-1}}\\
K_{2j} ({\mathcal O}_{E, S}; \, \Z/l^k ) \quad
\ar[r]^{r_{w^{\sigma_{{\bf a}}^{-1}}}} & \quad
\, K_{2j} (k_{w^{\sigma_{{\bf a}}^{-1}}}; \, \Z/l^k)}
\label{diagram 7.7}$$
The above diagram gives the following equality:
\begin{equation}\label{reduction1}  r_{w^{\sigma_{{\bf a}}^{-1}}}(\beta_{k}^{\ast \, n-m}) =
r_{w^{\sigma_{{\bf a}}^{-1}}}((\beta_{k}^{\ast \, n-m})^{ N {\bf a}^{n-m}\sigma_{{\bf a}}^{-1}})
= (r_w (\beta_{k}^{\ast \, n-m}))^{ N {\bf a}^{n-m} \sigma_{{\bf a}}^{-1}}.\end{equation}
For any $j \in \Z$, we have the following commutative diagram:
$$\xymatrix{
H^0 ({\mathcal O}_{E, S}; \, \Z/l^k (j)) \quad \ar@<0.1ex>[d]^{\sigma_{{\bf a}}^{-1}}
\ar[r]^{r_w} & \quad
H^0 (k_{w}; \, \Z/l^k (j))
\ar@<0.1ex>[d]^{\sigma_{{\bf a}}^{-1}}\\
H^{0} ({\mathcal O}_{E, S}; \, \Z/l^k (j) ) \quad
\ar[r]^{r_{w^{\sigma_{{\bf a}}^{-1}}}} & \quad
\, H^0 (k_{w^{\sigma_{{\bf a}}^{-1}}}; \, \Z/l^k (j))}
\label{diagram 7.71}$$
If $\xi_{l^k} := exp (\frac{2 \, \pi \, i}{l^k})$ is the generator of $\mu_{l^k}$ then the
above diagram gives
\begin{equation}\label{reduction2}  r_{w^{\sigma_{{\bf a}}^{-1}}}(\xi_{l^k}^{\otimes \, n-m}) =
r_{w^{\sigma_{{\bf a}}^{-1}}}(\xi_{l^k}^{\otimes \, n-m})^{ N {\bf a}^{n-m}\sigma_{{\bf a}}^{-1}})
= (r_w (\xi_{l^k}^{\otimes \, n-m}))^{ N {\bf a}^{n-m} \sigma_{{\bf a}}^{-1}}.\end{equation}

We can write the $m$--th Stickelberger element as follows
\begin{equation}\label{theta-m} \Theta_{m} ({\bf b}, {\bf f}_k) =
{\sum_{{\bf a} \, {\rm{mod}} \, {\bf f}_k}}' \,\, \left (\sum_{{\bf c} \, {\rm{mod}} \, {\bf f}_k, \,
w^{\sigma_{{\bf c}^{-1}}} = w}\,
\Delta_{m+1} ({\bf a} {\bf c}, {\bf b}, {\bf f}) \sigma_{{\bf c}^{-1}}\right )\cdot \sigma_{{\bf a}^{-1}}\,,\end{equation}
where ${\sum'_{{\bf a} \, \rm{mod} \, {\bf f}_k}}$ denotes the sum over a maximal set $\mathcal S$
of ideal classes ${\bf a} \mod {\bf f}_k$,  such that the primes $w^{\sigma_{{\bf a}}^{-1}}$, for ${\bf a}\in\mathcal S$,  are
distinct. By formula \eqref{2.9}, for every $m \geq 1$ and $n \geq 1$ we have
$$ \Delta_{n+1} ({\bf a}, {\bf b}, {\bf f}) \equiv
N {\bf a}^{n- m} \, N {\bf b}^{n- m} \, \Delta_{m+1} ({\bf a} {\bf c}, {\bf b}, {\bf f})
\mod \,\, w_{\text{min} \, \{m, n \}}(K_{{\bf f}}).$$
It is clear that for all $m,n \geq 1$ we get the following
congruence $\mod
w_{\text{min} \, \{m, n \}}(K_{{\bf f}_k}).$

$$\Theta_{n} ({\bf b}, {\bf f}_k)\equiv$$
$${\sum_{{\bf a} \, {\rm{mod}} \, {\bf f}_k}}' \,\, (\sum_{{\bf c} \, {\rm{mod}} \, {\bf f}_k, \,
w^{\sigma_{{\bf c}^{-1}}} = w}\,
N {\bf a}^{n-m} \, N {\bf c}^{n-m}\, N {\bf b}^{n-m} \, \Delta_{m+1} ({\bf a} {\bf c}, {\bf b},
{\bf f}_k)
\sigma_{{\bf c}^{-1}}) \sigma_{{\bf a}^{-1}}$$

\noindent Equalities (\ref{reduction1}), (\ref{reduction2}), (\ref{theta-m}), Lemma
\ref{Lemma 7.2}, Gillet's result \cite{Gi} (diagram \eqref{Gillet}), the compatibility of
$t (n-m)$ and $t_v (n-m)$ with $\partial$ and
the above congruences satisfied by Stickelberger elements lead in both cases $n- m \geq 0$
and $n - m < 0$ to the following equalities.

\begin{eqnarray}\nonumber
\partial_{E} ( x_{w, k}^{\Theta_{m} ({\bf b}, {\bf f}_k)} \ast
\beta_{k}^{\ast \, n-m})^{N {\bf b}^{n-m}}= \\
 \nonumber
 ={\sum_{{\bf a} \, {\rm{mod}} \, {\bf f}_k}}'
\,\, \xi_{w, k}^{\sum_{{\bf c} \, {\rm{mod}} \, {\bf f}_k, \, w^{\sigma_{{\bf c}^{-1}}} = w} \,
\Delta_{m+1} ({\bf a} {\bf c}, {\bf b}, {\bf f}_k) \sigma_{({\bf a} {\bf c})^{-1}}}
\ast (\beta_{k}^{{\ast \, n-m}})^{(N {\bf a} {\bf c})^{n-m} N {\bf b}^{n-m}
\sigma_{(ac)^{-1}}}= \\
\nonumber
 =(\xi_{w, k}
\ast \beta_{k}^{{\ast \, n-m}})^{
\sum'_{{\bf a} \, {\rm{mod}} \, {\bf f}_k} \sum_{{\bf c} \, {\rm{mod}} \, {\bf f}_k, \,
w^{\sigma_{{\bf c}^{-1}}} = w} \, \Delta_{m+1} ({\bf a} {\bf c}, {\bf b}, {\bf f}_k)
(N {\bf a} {\bf c})^{n-m} N {\bf b}^{n-m} \sigma_{({\bf a} {\bf c})^{-1}}} = \\
\nonumber
=(\xi_{w, k} \ast \beta_{k}^{{\ast \, n-m}})^{\Theta_{n} ({\bf b}, {\bf f}_k)}.
\end{eqnarray}
By the first commutative diagram of this proof, the equalities above and
Lemma \ref{Lemma 2.1}, we obtain:
$$ \partial_{F} (\blambda_{v, l^k}) =
N (\partial_{E} ( x_{w, k}^{\Theta_{m} ({\bf b}, {\bf f}_k)} \ast
\beta_{k}^{\ast \, n-m})^{N{\bf b}^{n-m}})^{\gamma_l} =
N ((\xi_{w, k} \ast \beta_{k}^{{\ast \, n-m}})^{\Theta_{n} ({\bf b}, {\bf f}_k)})^{\gamma_l} = $$
$$ = (N (\xi_{w, k} \ast \beta_{k}^{{\ast \, n-m}}))^{{\gamma_l}^{-1}
\Theta_{n} ({\bf b}, {\bf f}) {\gamma_l}}
= (N (\xi_{w, k} \ast \beta_{k}^{{\ast \, n-m}}))^{\Theta_{n} ({\bf b}, {\bf f})}.$$
\end{proof}

\begin{theorem}\label{Theorem 7.8} For every $v$ such that
$l\mid q_{v}^n - 1$ and for all $k \geq k(v),$ there are homomorphisms
$$ \Lambda_{v, \, l^k} \, : \, K_{2n-1} (k_v ;\, \Z/l^k) \rightarrow
K_{2n} ({\mathcal O}_{F, S} ;\, \Z/l^k),$$
$$ \Lambda_{v, \, l^k}^{et} \, : \, K_{2n-1}^{et} (k_v ;\, \Z/l^k) \rightarrow
K_{2n}^{et} ({\mathcal O}_{F, S} ;\, \Z/l^k):$$
which satisfy the following equalities:
$$\Lambda_{v, \, l^k} \, (N(\xi_{w, k} \ast \beta_{k}^{\ast \, n-m}))
\,\, = \,\, \blambda_{v, l^k},$$
$$\Lambda_{v, \, l^k}^{et} \, (N(\zeta_{w, k} \ast \beta_{k}^{\ast \, n-m}))
\,\, = \,\, \blambda^{et}_{v, l^k}.$$

\end{theorem}

\begin{proof}
The definition of $\Lambda_{m}$ (see (\ref{firststepsplitting})),
combined with the natural
isomorphism $K_{2m-1}(k_w) / l^k \cong K_{2m-1}(k_w; \Z/l^k\Z)$ and the natural
monomorphism
$$K_{2m}({\mathcal O}_{E, S}) / l^k \to K_{2m}({\mathcal O}_{E, S}; \Z/l^k\Z),$$
coming from the corresponding Bockstein exact sequences, leads to the following homomorphism:
$$\widetilde \Lambda_{m} : K_{2m-1}(k_w; \Z/l^k\Z) \to  K_{2m}({\mathcal O}_{E,S}; \Z/l^k\Z).$$
Multiplying on the target and on the source of this homomorphism with the $n-m$ power
of the Bott element if $n-m \geq 0$ (resp. applying the map $t_w (n-m)$ to the source and
$t (n-m)$ to the target if $n -m < 0$) under the observation that the following map is an
isomorphism:
$$\xymatrix
{K_{2m-1}(k_w; \Z/l^k\Z)\ar[r]_{\sim}^{{\ast\beta_k^{\ast n-m}}} &K_{2n-1}(k_w; \Z/l^k\Z)}$$
(cf. the notation of $t (j)$ and $t_w (j)$ )
show that there exists a unique homomorphism
$$\widetilde \Lambda_{m} \ast \beta_{k}^{\ast \, n-m}\, :\,K_{2n-1} (k_w ;\, \Z/l^k) \rightarrow
K_{2n} ({\mathcal O}_{E, S} ;\, \Z/l^k),$$
sending  $ \xi_{w, k} \ast \beta_{k}^{\ast \, n-m} \rightarrow
x_{w, k}^{\Theta_{m} ({\bf b}, {\bf f}_{k})} \ast \beta_{k}^{\ast \, n-m}$.
Next, we compose the homomorphisms $\widetilde \Lambda_{m} \ast \beta_{k}^{\ast \, n-m}$
defined above and
$$Tr_{E/F}\, : \, K_{2n} ({\mathcal O}_{E, S} ;\, \Z/l^k) \rightarrow
K_{2n} ({\mathcal O}_{F, S} ;\, \Z/l^k)$$
to obtain the following homomorphism:
$$ Tr_{E/F}\circ (\widetilde \Lambda_{m} \ast \beta_{k}^{\ast \, n-m}):
K_{2n-1} (k_w ;\, \Z/l^k) \rightarrow
K_{2n} ({\mathcal O}_{F, S} ;\, \Z/l^k)\,.$$
We observe that this homomorphism factors through the quotient of $G(k_w/k_v)$--coinvariants
$$K_{2n-1} (k_w ;\, \Z/l^k)_{G(k_{w} / k_{v})}:=
K_{2n-1} (k_w ;\, \Z/l^k) / K_{2n-1} (k_w ;\, \Z/l^k)^{Fr_v - Id},$$
where $Fr_{v} \in G(k_{w} / k_{v})\subseteq G(E/F)$ is the Frobenius element
of the prime $w$ over $v.$ Since $Fr_v$ acts via $q_v^n$--powers on $K_{2n-1}(k_w)$, the canonical isomorphism
$K_{2n-1} (k_w ;\, \Z/l^k)\cong K_{2n-1}(k_w)/l^k$ (see \S3) and assumption $k\geq k(v)$ give
$$K_{2n-1} (k_w ;\, \Z/l^k)_{G(k_{w} / k_{v})} \cong K_{2n-1} (k_w ;\, \Z/l^k) / l^{k(v)}\cong
K_{2n-1} (k_w)/l^{k(v)}.$$
The obvious commutative diagram with surjective vertical morphisms (see \S 3)
$$\xymatrix{
K_{2n-1} (k_w)/l^k \quad \ar@<0.1ex>[d]^{N_{w /v}}
\ar[r]^{\cong}  \quad &
\quad K_{2n-1} (k_w ;\, \Z/l^k)
\ar@<0.1ex>[d]^{N_{w /v}}\\
K_{2n-1} (k_v)/l^k \quad
\ar[r]^{\cong} &
\quad \, K_{2n-1} (k_v ;\, \Z/l^k)}
\label{diagram 2.5}$$
combined with the last isomorphism above, gives an isomorphism
$$\xymatrix{K_{2n-1} (k_w ;\, \Z/l^k)_{G(k_{w} / k_{v})}
\ar[r]^{\quad N_{w /v}}_{\quad\sim} &
K_{2n-1} (k_v ;\, \Z/l^k)}$$
Now, the required homomorphism is:
\begin{equation}
\xymatrix{\Lambda_{v, \, l^k}
\, : \, K_{2n-1} (k_v ; \, \Z/l^k) \ar[r] & K_{2n} ({{\mathcal O}_{F, S}} ; \, \Z/l^k) }
\label{tag 4.6}
\end{equation}
defined by
$$\Lambda_{v, \, l^k} (x):=
[Tr_{E/F}\circ (\widetilde \Lambda_{m} \ast \beta_{k}^{\ast \, n-m})\circ {N_{w /v}}^{-1}(x)]^{N{\bf b}^{n-m} \gamma_l}\,,$$
for all $x\in K_{2n-1} (k_v ; \, \Z/l^k)$.
By definition, this map sends $N(\xi_{w, k} \ast \beta_{k}^{\ast \, n-m})$ onto the element
$\blambda_{v, l^k}:=
Tr_{E/F}(x_{w, k}^{\Theta_{m} ({\bf b}, {\bf f}_{k})}
\ast \beta_{k}^{\ast \, n-m})^{N{\bf b}^{n-m} \gamma_l}$.
\end{proof}
\bigskip

\subsection{Constructing $\Lambda$ and $\Lambda^{et}$ for $K$--theory without coefficients}  Let us fix $n>0$. In this section,
we use the special elements and $\blambda_{v, l^k}$ and $\blambda^{et}_{v, l^k}$ defined above
to construct the maps $\Lambda_n$ and $\Lambda^{et}_n$ for the $K$--theory (respectively
\'etale $K$--theory) without coefficients. Since $n$ is fixed throughout, we will denote $\Lambda:=\Lambda_n$ and
$\Lambda^{et}:=\Lambda^{et}_n$.
\medskip

\noindent Observe that for every $j > 0$ and every prime $l$,
the Bockstein exact sequence \eqref{Bokstein} and results of Quillen
\cite{Q2}, \cite{Q3} give
natural isomorphisms
\begin{equation}
K_{j} ({\mathcal O}_{F, S})_l \,\, \cong \,\,
\varprojlim_{k} K_{j} ({\mathcal O}_{F, S} ;\, \Z/l^k),
\label{invlimringofintegers}\end{equation}
\begin{equation}
K_{j} (k_v)_l \,\, \cong \,\,
\varprojlim_{k} K_{j} (k_v ;\, \Z/l^k)\,.
\label{invlimfinitefields}\end{equation}
Similar isomorphisms hold for the \' etale $K$-theory.

\begin{definition} We define $\blambda_v
\in K_{2n} ({\mathcal O}_{F, S})_l$
and $\blambda_v^{et}
\in K_{2n}^{et} ({\mathcal O}_{F, S})$
to be the elements corresponding to
$$(\blambda_{v, l^k})_{k} \in
\varprojlim_{k} K_{2n} ({\mathcal O}_{F, S} ;\, \Z/l^k), \qquad (\blambda^{et}_{v, l^k})_{k} \in
\varprojlim_{k} K_{2n}^{et} ({\mathcal O}_{F, S} ;\, \Z/l^k)$$
via the isomorphism (\ref{invlimringofintegers})
and its \'etale analogue, respectively.\end{definition}

\begin{definition} We define $\xi_v \in K_{2n-1} (k_v)_l$ and
$\zeta_v \in K_{2n-1}^{et} (k_v)$ to be the elements corresponding to
$$(N(\xi_{w, k} \ast \beta_{k}^{\ast \, n-m}))_k \in
\varprojlim_{k} K_{2n-1} (k_v ;\, \Z/l^k), $$
$$(N(\zeta_{w, k} \ast \beta_{k}^{\ast \, n-m}))_k \in
\varprojlim_{k} K_{2n-1}^{et} (k_v ;\, \Z/l^k),$$
via the isomorphism (\ref{invlimfinitefields}) and its \'etale analogue, respectively.
\end{definition}
\begin{definition} Assume that $l\mid q_v^n-1$. Since the homomorphisms $\Lambda_{v, \, l^k},$ and $\Lambda^{et}_{v, \, l^k},$
from Theorem \ref{Theorem 7.8}, are compatible with
the coefficient reduction maps $r_{k^{\prime}/k}$, for all \, $k'\geq k\geq k(v),$ \,
we can define homomorphisms
$$ \Lambda_{v} := \varprojlim_{k} \Lambda_{v, \, l^k}:
K_{2n-1} (k_v)_l \rightarrow K_{2n} ({\mathcal O}_{F, S})_l\hookrightarrow  K_{2n} (F)_l,
$$
$$ \Lambda_{v}^{et} := \varprojlim_{k} \Lambda^{et}_{v, \, l^k}:
K_{2n-1}^{et} (k_v) \rightarrow K_{2n}^{et} ({\mathcal O}_{F, S})\hookrightarrow  K^{et}_{2n} (F)_l,
$$
for all $v.$ Here, the rightmost arrows are the inclusions $K_{2n}({\mathcal O}_{F, S}) \subset K_{2n} (F)$
and $K_{2n}^{et}({\mathcal O}_{F, S}) \subset K_{2n}^{et} (F)_l$, respectively.
If $l\nmid q_v^n-1$, then the morphisms $\Lambda_{v}$ and $\Lambda_{v}^{et}$ are trivial, by default.
\end{definition}

 \begin{remark}\label{xi-lambda} It is clear from
Theorem \ref{Theorem 7.8} that, for all $v$, we have
$$\Lambda_{v} (\xi_v) = \blambda_v, \qquad \Lambda_{v}^{et} (\zeta_v) = \blambda_v^{et}.$$
\end{remark}

\begin{definition}\label{Definition 7.9}
We define the maps $\Lambda_n$ and $\Lambda_n^{et}$ as follows:
$$\Lambda \, : \, \bigoplus_{v} K_{2n-1} (k_v)_l
\, \rightarrow K_{2n} (F)_l, \qquad \Lambda \, := \, \prod_{v} \Lambda_{v}.$$
$$\Lambda^{et} \, : \, \bigoplus_{v} K_{2n-1}^{et} (k_v)
\, \rightarrow K_{2n}^{et} (F)_l, \qquad
\Lambda^{et} \, := \, \prod_{v} \Lambda_{v}^{et}.$$
\end{definition}

\begin{theorem}\label{Theorem 7.10}
The maps $\Lambda$ and $\Lambda^{et}$ satisfy the following properties:
$$\partial_{F} \circ \Lambda (\xi_{v}) =
\xi_{v}^{\Theta_{n} ({\bf b}, {\bf f})},
$$
$$\partial_{F}^{et} \circ \Lambda^{et} (\zeta_{v}) =
\zeta_{v}^{\Theta_{n} ({\bf b}, {\bf f})}.
$$
\end{theorem}
\begin{proof}
Consider the following commutative diagram.
$$\xymatrix{
K_{2n} ({\mathcal O}_{F, S})/l^k \quad \ar@<0.1ex>[d]^{}
\ar[r]^{\bigoplus_{v \in S} \, \partial_{v}} \quad &
\quad\quad
\bigoplus_{v \in S}  K_{2n-1} (k_v)/l^k
\ar@<0.1ex>[d]^{}\\
K_{2n} ({\mathcal O}_{F, S} ;\, \Z/l^k) \quad
\ar[r]^{\bigoplus_{v \in S} \, \partial_{v}} &
\quad\quad\quad \bigoplus_{v \in S} \, K_{2n-1} (k_v ;\, \Z/l^k)}
\label{diagram 5.20}$$
The vertical arrows in the diagram come from the Bockstein
exact sequence. It is clear from the diagram that the inverse limit
over $k$ of the bottom horizontal arrow
gives the boundary map $\partial_{F} =
\bigoplus_{v \in S} \partial_{v}:$
$$
\partial_{F} \, :\, K_{2n} ({\mathcal O}_{F, S})_l  \rightarrow
\bigoplus_{v}  K_{2n-1} (k_v)_l.
$$
Now, the theorem follows by Theorems \ref{Theorem 7.7} and
\ref{Theorem 7.8}.
\end{proof}

\noindent
In the next proposition we will construct a Stickelberger splitting map $\Gamma$
which is complementary to the map $\Lambda$ constructed above.
\begin{equation}\label{exact-sequence} 0 \stackrel{}{\longrightarrow} K_{2n} ({\mathcal O}_{F})_l
{{\stackrel{i}{\longrightarrow}} \atop {\stackrel{\Gamma}{\longleftarrow}}}
K_{2n} (F)_l
{{\stackrel{\partial_F}{\longrightarrow}} \atop {\stackrel{\Lambda}{\longleftarrow}}}
\bigoplus_{v} K_{2n-1} (k_v)_l
\stackrel{}{\longrightarrow} 0.\end{equation}
 The existence of $\Gamma$ was suggested in 1988 by Christophe Soul\'e in a letter to the
first author and it is a direct consequence of the following
module theoretic lemma.

\begin{lemma}\label{modules-1} Let $R$ be a commutative ring with $1$ and let $r\in R$ be fixed. Let
$$\xymatrix{0\ar[r] &A\ar[r]^{\iota} &B\ar[r]^\pi &C\ar[r] &0}
$$
be an exact sequence of $R$--modules. Then, the following are
equivalent:
\begin{enumerate}
\item There exists an $R$--module morphism $\Lambda: C\to B$,
    such that $\pi\circ\Lambda=r\cdot{id}_C$.
\item There exists an $R$--module morphism $\Gamma: B\to A$,
    such that $\Gamma\circ\iota=r\cdot{id}_A$.
\end{enumerate}
Moreover, if $\Lambda$ and $\Gamma$ exist, they can be chosen so
that $\Gamma\circ\Lambda=0$.
\end{lemma}

\begin{proof}[Proof of Lemma] Assume that (1) holds. By the defining property
of $\Lambda$, we have
\begin{equation}\label{Lambda-property}(\Lambda\circ\pi)(-b)+rb\in {\rm Im}(\iota), \qquad \forall\, b\in B.
\end{equation}
We define $\Gamma(b):=\iota^{-1}((\Lambda\circ\pi)(-b)+rb)$, for
all $b\in B$, where $\iota^{-1}(x)$ is the preimage of $x$ via
$\iota$, for all $x\in{\rm Im}(\iota)$. One can check without
difficulty that $\Gamma$ is an $R$--module morphism which
satisfies
$$\Gamma\circ\iota = r\cdot id_A, \qquad \Gamma\circ\Lambda =0.$$

Now, assume that (2) holds. Let $c\in C$. Take $b\in B$, such that
$\pi(b)=c$. Then, by the defining property of $\Gamma$, one can
check that the element $(\iota\circ\Gamma)(-b)+rb\in B$ is
independent on the chosen $b$. For all $c\in C$, we define
$\Lambda(c):= (\iota\circ\Gamma)(-b)+rb$, where $b\in B$, such
that $\pi(b)=c$. It is easily seen that the map $\Lambda$ defined
this way is an $R$--module morphism and it satisfies
$$\pi\circ\Lambda=r\cdot id_C, \qquad \Gamma\circ\Lambda=0.$$
\end{proof}

\begin{proposition}\label{Proposition 7.11} The existence of a map
$\Lambda $ satisfying the property
$(\partial_{F} \circ \Lambda )(\xi_{v}) =
\xi_{v}^{\Theta_{n} ({\bf b}, {\bf f})}
$
is equivalent to the existence of a map
$\Gamma \, : \, K_{2n} (F)_l \rightarrow K_{2n} ({\mathcal O}_{F})_l$
with the property
$(\Gamma \circ i) (\eta) =
\eta^{\Theta_{n} ({\bf b}, {\bf f})}.$ Moreover, if they exist, the maps $\Lambda$ and $\Gamma$ can be chosen so that $\Gamma \circ \Lambda = 0.$
\end{proposition}
\begin{proof}
The proof of the Proposition follows directly from the above Lemma
applied to $R:=\Bbb Z[G(F/K)]$, $r:=\Theta_{n} ({\bf b}, {\bf f})$
and Quillen's localization exact sequence \eqref{exact-sequence}.
\end{proof}

\begin{remark}
From the proof of Lemma \ref{modules-1} it is clear that if one of
the maps $\Lambda$ and $\Gamma$ is given, then the other one can
be chosen such that $$r\cdot id_B=\Lambda\circ\pi
+i\circ\Gamma.$$\end{remark}

\begin{remark}\label{Remark 7.12}
Observe that the map $\Lambda$ is defined in the same way for both cases
$l \nmid n$ and $l \mid n.$ If restricted to the particular case
$K=\Bbb Q$, our construction improves upon that of \cite{Ba1}. In loc.cit., in the case $l\mid n$
the map $\Lambda$ was constructed only up to a factor of $l^{v_l(n)}$.
\end{remark}
\medskip

 Analogously,
there is a Stickelberger splitting map $\Gamma^{et}$ which is
complementary to the map $\Lambda^{et}$ such that the \' etale
analogue of the Proposition \ref{Proposition 7.11} holds.
\medskip

\subsection{Annihilating $div\, K_{2n}(F)_l$}
Now, let us give a set of immediate applications of our
construction of the Stickelberger splitting maps $\Lambda_n$. In what follows, if $A$ ia an abelian group, $div\,A$ denotes its subgroup
of divisible elements.
The applications which follow concern annihilation of the groups $div\,
K_{2n}(F)_l$ by higher Stickelberger elements of the type proved
in \cite{Ba1} in the case where the base field is $\Bbb Q$. The difference is that while \cite{Ba1} deals with
abelian extensions $F/\Bbb Q$, under certain restrictions if
$l\mid n$, we deal with abelian extensions $F/K$ of an arbitrary
totally real number field $K$ under no restrictive conditions. The desired annihilation result follows from the following.

\begin{lemma}\label{modules-2} With notations as in Lemma \ref{modules-1}, assume that a map $\Lambda$ exists.
Further, assume that $C$ and $A$ viewed as abelian groups satisfy $div\,C=0$ and $A$ has finite exponent.
Then, we have the following.
\begin{enumerate}
\item $div\,B\subseteq {\rm Im}(\iota)$.
\item $r$ annihilates $div\,B$.
\end{enumerate}
\end{lemma}
\begin{proof}
Since any morphism maps divisible elements to divisible elements, we have $\pi(div\, B)=0$, by our assumption on $C$. 
This concludes the proof of (1).

Let $x\in div\, B$. Let $m$ be the exponent of $A$ and let $b\in B$, such that $m\cdot b=x$. Multiply \eqref{Lambda-property} by $m$ to conclude that
$$(\Lambda\circ\pi)(-x)+r\cdot x =0.$$
Now, part (1) implies that $\pi(x)=0$. Consequently, the last equality implies that $r\cdot x=0$, which concludes the proof.
\end{proof}
\begin{theorem}\label{Theorem 7.13}
Let $m > 0$ be a natural number.
Assume that the Stickelberger elements $\Theta_{m} ({\bf b}, {\bf f}_k)$ annihilate
the groups $K_{2m} ({\mathcal O}_{F_k})_l$ for all $k \geq 1.$
Then the Stickelberger's element $\Theta_{n} ({\bf b}, {\bf f})$ annihilates the group
$div \, K_{2n} (F)_l$ for every $n \geq 1.$
\end{theorem}
\begin{proof}
The proof is very similar to that of [Ba1, Cor. 1, p. 340].
Let us fix $n\geq 1$.  Under our annihilation hypothesis, we have constructed a map
$\Lambda:=\Lambda_n$ satisfying the properties in Proposition \ref{Proposition 7.11}
relative to the Quillen localization sequence \eqref{exact-sequence}. Note that
$A:=K_{2n}(O_F)_l$ is finite and therefore it has a finite exponent. Also, note that
$C:=\bigoplus_{v} K_{2n-1} (k_v)_l$ is a direct sum of finite abelian groups and therefore
$div\, C=0$. Consequently, the exact sequence \eqref{exact-sequence} together with the map
$\Lambda$ and element $r:=\Theta_{n} ({\bf b}, {\bf f})$ in the ring $R:=\Bbb Z[G(F/K)]$
satisfy the hypotheses of Lemma \ref{modules-2}. Therefore, we have
$$\Theta_{n} ({\bf b}, {\bf f})\cdot div\, K_{2n}(F)_l=0.$$
\end{proof}

\noindent
\begin{remark}\label{second proof of annihilation of div}
Observe that we can restrict the map $\Lambda$ to the $l^k$--torsion part, for
any $k \geq 1.$ For any $k \gg 0$, there is an exact sequence
$$
0 \stackrel{}{\longrightarrow} K_{2n} ({\mathcal O}_{F}) [l^k]
\stackrel{}{\longrightarrow} K_{2n} (F) [l^k]
{{\stackrel{\partial_F}{\longrightarrow}} \atop
{{\stackrel{\Lambda}{\longleftarrow}}}} \bigoplus_{v} K_{2n - 1} (k_v) [l^k]
\stackrel{}{\longrightarrow} div
(K_{2n} (F)_l) \stackrel{}{\longrightarrow} 0
$$
By Theorem \ref{Theorem 7.10}, we know that $\partial_F \circ \Lambda$
is the multiplication by $\Theta_{n} ({\bf b}, {\bf f}).$
As pointed out in the Introduction, this implies the annihilation of
$div\,( K_{2n} (F)_l)$ and consequently gives a second proof for Theorem
\ref{Theorem 7.13}
\end{remark}
\bigskip

\noindent
Let us define $F_0 := F$ and:
$$
\Theta_{n} ({\bf b}, {\bf f}_{0}) \,\, = \,\,
\left\{
\begin{array}{lll}
\bigl( \, \prod_{{{\bf l}\nmid{\bf f}} \atop
{{\bf l} \, | \, l}} \, (1 - ({\bf l}, \, F)^{-1} N{\bf l}^{n}) \, \bigr) \,\,
\Theta_{n} ({\bf b}, {\bf f}) &\rm{if}&l \nmid {\bf f}\\
\Theta_{n} ({\bf b}, {\bf f}) &\rm{if}& l \mid {\bf f}\\
\end{array}\right.
$$
Hence by the formula (\ref{2.53}) we get
\begin{equation}
Res_{F_{k+1} / F_{k}} \,\, \Theta_{n} ({\bf b}, {\bf f}_{k+1})
=  \,\, \Theta_{n} ({\bf b}, {\bf f}_k)
\label{RestofStickinIwasawaTower}\end{equation}
Hence by formula (\ref{RestofStickinIwasawaTower}) we can define the element
\begin{equation}
\Theta_{n} ({\bf b}, {\bf f}_{\infty}) :=
\varprojlim_{k} \Theta_{n} ({\bf b}, {\bf f}_{k}) \in \varprojlim_{k}\, \Z_l [G(F_k / F)].
\label{InfiniteStickelberger}\end{equation}

\begin{corollary}\label{Corollary 7.15}
Let $m > 0$ be a natural number. Assume that the Stickelberger elements
$\Theta_{m} ({\bf b}, {\bf f}_k)$ annihilate
the groups $K_{2 m} ({\mathcal O}_{F_k})_l$ for all $k \geq 1.$
Then the Stickelberger element $\Theta_{n} ({\bf b}, {\bf f}_k)$ annihilates the group
$div \, K_{2n} (F_k)_l$ for every $k \geq 0$ and every $n \geq 1.$
In particular $\Theta_{n} ({\bf b}, {\bf f}_{\infty})$ annihilates the group
$\varinjlim_{k} div \, K_{2n} (F_{k})_l$ for every $n \geq 1.$
\end{corollary}
\begin{proof}
Follows immediately from Theorem \ref{Theorem 7.13}.
\end{proof}

\begin{theorem}\label{Theorem 7.18}
Let $F/K$ be an abelian CM extension of an arbitrary totally real
number field $K$ and let $l$ be an odd prime. If the Iwasawa
$\mu$--invariant $\mu_{F,l}$ associated to $F$ and $l$ vanishes,
then $\Theta_{n} (\bf b, \bf f)$ annihilates the group $div
(K_{2n} (F)_l)$, for all $n \geq 1$ and all ${\bf b}$ coprime to
$w_{n+1}(F){\bf f}l$.
\end{theorem}

\begin{proof} In \cite{GP} (see Theorem 6.11), it is shown that if $\mu_{F, l}=0$, then  $\Theta_{n} (\bf b, \bf f)$
annihilates $K_{2n}^{et}(O_F[1/l])$, for all $n\geq 1$ and all
${\bf b}$ as above. From the definition of Iwasawa's
$\mu$--invariant one concludes right away that if $\mu_{F, l}=0$,
then $\mu_{F_k, l}=0$, for all $k$. Consequently, $\Theta_{1} (\bf
b, \bf f_k)$ annihilates $K_{2}^{et}(O_{F_k}[1/l])$, for all $k$.
Now, one applies Tate's Theorem \ref{Tate's Theorem} to conclude
that $\Theta_{1} (\bf b, \bf f)$ annihilates $K_{2}(O_{F_k})_l$,
for all $k$. Theorem \ref{Theorem 7.13} implies the desired
result.
\end{proof}

\begin{remark}
It is a classical conjecture of Iwasawa that $\mu_{F, l}=0$, for
all number fields $F$ and all primes $l$.
\end{remark}

\begin{corollary}
Let $F/\Q$ be an abelian extensions of conductor $f.$ Then
$\Theta_{n} (b, f)$ annihilates the group $div \, K_{2n} (F)_l$
for all $n \geq 1$ and all ${\bf b}$ coprime to
$w_{n+1}(F){\bf f}l$.
\end{corollary}

\begin{proof}
By a well known theorem of Ferrero-Washington and Sinnott,
$\mu_{F, l}=0$ for all fields $F$ which are abelian extensions of
$\Bbb Q$ and all $l$. Now, the Corollary is an immediate
consequence of the previous Theorem.
\end{proof}

\begin{remark}
Observe that the Corollary above strengthens Corollary 1, p. 340
of \cite{Ba1} in the case $l \, | \,n.$
\end{remark}


\noindent
\section{Constructing Euler systems out of $\Lambda$--elements}
As mentioned in \S4, in this section we construct Euler systems
for the even $K$-theory of CM abelian extensions of totally real number fields.
We construct an Euler system in the $K$--theory with coefficients. Then,
by passing to a projective limit, we obtain an Euler system in Quillen $K$--theory. Our constructions are quite
different from those in \cite{BG1}, where an Euler systems in the odd $K$--theory with finite coefficients
of CM abelian extensions of $\Bbb Q$ was described.
\medskip

As above, we fix a finite abelian CM extension $F/K$ of a totally real number field of conductor $\bf f$ and fix a prime number $l$.
We let ${\bf L} = {\bf l}_1 \dots {\bf l}_t$ be a product of mutually distinct
prime ideals of ${\mathcal O}_K$ coprime to $l \, {\bf f}.$
We let
$F_{\bf L} := F K_{\bf L}$, where
$K_{\bf L}$ is the ray class field of $K$ for the ideal ${\bf L}.$
Since $F/K$ has conductor ${\bf f}$ the CM-extension $F_{\bf L} / K$
has conductor dividing ${\bf L} {\bf f}.$ As usual, we let $F_{{\bf L}{l^k}} := F_{{\bf L}}(\mu_{l^k})$, for every $k \geq 0.$
\medskip

Let us fix a prime $v$ in ${\mathcal O}_{F}$ sitting above a rational prime $p \ne l$.
Let $S := S_v$ be the set consisting of all the primes of ${\mathcal
O}_{F}$ sitting above $p$ or above $l$.
For all $\bf L$ as above and $k\geq 0$, we fix primes $w_{k}(\bf L)$ of $O_{F_{{\bf L}l^k}}$ sitting above
$v$, such that $w_{k'}(\bf L')$ sits above $w_{k}(\bf L)$ whenever $l^k{\bf L}\mid l^{k'}{\bf L'}$. For simplicity,
we let $v({\bf L}):=w_{0}({\bf L})$, for all $\bf L$ as above. Also, if $k$ is fixed, we let $w({\bf L}):=w_k({\bf L})$, for all ${\bf L}$ as above.

\medskip

Let us fix integers $m  > 0$.
For all ${\bf L}$ as above and all $k\geq 0$, let $\Theta_{m} ({\bf b}_{\bf L},\, {\bf L}{\bf f}_k )$ denote the $m$-th
Stickelberger element for the integral ideal
${\bf b}_{\bf L}$ of $O_F$, coprime to ${\bf Lf}l$, and the extension $F_{{\bf L}{l^{k}}} / K$.
As usual, we assume throughout that $\Theta_{m} ({\bf b}_{{\bf L}}, {\bf L} {\bf f}_k)$ annihilates
$K_{2m} ({\mathcal O}_{F_{{\bf L}{l^k}}})$, for all  ${\bf L}$ as above and all $k\geq 0$.

By the surjectivity of the transfer maps for the $K$-theory of finite fields
we can  fix generators $\xi_{w_k({{\bf L}}), k}$ of
$K_{2m-1}(k_{w_k({{\bf L}})})_l$, for all $k\geq 0$ and ${\bf L}$ as above, such that
$$N_{{w_{k'}({{\bf L}})}/w_k({{\bf L}})}
(\xi_{w_{k'}({{\bf L}}),
k'})= \xi_{w_k({{\bf L}}), k}\,,$$
whenever we have $k\leq k'$.

\begin{remark}\label{invertible-factor} Note that the cyclicity of the groups $K_{2m-1}(k_{w_k({{\bf L}})})_l$ and the surjectivity of the appropriate
transfer maps implies that the elements
$$(\xi_{w_k({\bf L}), k})_k, \qquad  (N_{w_k({\bf L'})/w_k({\bf L})}(\xi_{w_k({\bf L'}), k}))_k,$$
viewed inside of $\Bbb Z_l$--module
$\varprojlim_{k} K_{2m-1}(k_{w_k({{\bf L}})})_l$, differ by a factor in $\Bbb Z_l^\times$, for all ${\bf L}$ and ${\bf L'}$ as above,
such that ${\bf L}|{\bf L'}$. Above, the projective limit is taken with respect to the transfer maps.
\end{remark}
\medskip

Let us fix $k\geq 0$. For any ${\bf L}$ as above, we have the localization exact sequence:
$$
0 \stackrel{}{\longrightarrow} K_{2 m} ({\mathcal O}_{F_{{\bf L}{l^k}}})
\stackrel{}{\longrightarrow} K_{2 m} ({\mathcal O}_{F_{{\bf L}{l^k}}, \, S})
\stackrel{\partial}{\longrightarrow} \bigoplus_{v_0 \in S} \bigoplus_{w | v_0}
K_{2 m -1} (k_w) \stackrel{}{\longrightarrow} 0,
$$
where the direct sum is taken with respect to all the primes $w$ in ${\mathcal O}_{F_{{\bf L}{l^k}}}$
which sit above primes $v_0$ in $S$.
Pick an element $x_{w({\bf L}), k} \in K_{2 m} ({\mathcal O}_{F_{{\bf L} l^k}, S})_l$,
such that ${\partial} (x_{w ({\bf L}), k}) = \xi_{w ({\bf L}), k}.$ The following element:
\begin{equation}
\Lambda_{m} (\xi_{w ({\bf L}), k}) := x_{w({\bf L}), k}^{\Theta_{m} ({\bf b}_{{\bf L}}, {\bf L}{\bf f}_k)}
\label{ESwithNoCoeff}
\end{equation}
does not depend on the choice
of $x_{w({\bf L}), k}$ since $\Theta_{m} ({\bf b}_{{\bf L}}, {\bf L}{\bf f}_k)$ annihilates
$K_{2 m} ({\mathcal O}_{F_{{\bf L}{l^k}}}).$ Observe that by construction we have
the following equalities:

\begin{equation}
\partial_{F_{{\bf L} l^k}} (Tr_{F_{{\bf L}^{\prime} l^k}/F_{{\bf L} l^k}}
\bigl( x_{w ({\bf L}^{\prime}), k}\bigr)) =
N_{w ({\bf L}^{\prime})/w ({\bf L})} (\partial_{F_{{\bf L}^{\prime} l^k}} \bigl( x_{w ({\bf L}^{\prime}), k}\bigr) =
\label{ESauxiliary1}\end{equation}
\begin{equation}
= N_{w ({\bf L}^{\prime})/w ({\bf L})} ( \xi_{w ({\bf L}^{\prime}), k}),
\nonumber
\end{equation}

\begin{equation}
N_{w ({\bf L})/v ({\bf L})} (N_{w ({\bf L}^{\prime})/w ({\bf L})} ( \xi_{w ({\bf L}^{\prime}), k})) = N_{v ({\bf L}^{\prime})/v ({\bf L})} (N_{w ({\bf L}^{\prime})/v ({\bf L}^{\prime})} ( \xi_{w ({\bf L}^{\prime}), k})) =
\label{ESauxiliary2}
\end{equation}
\begin{equation}
=  N_{v ({\bf L}^{\prime})/v ({\bf L})} ( \xi_{v ({\bf L}^{\prime}), 0}))
\nonumber
\end{equation}
\medskip

\medskip

\noindent We choose the ideals ${\bf b}_{{\bf L}}$
in such a way so that they are coprime to $ l \, {{\bf L}} {{\bf f}}$ and
$$N {\bf b}_{{\bf L}^{\prime}} \equiv N {\bf b}_{{\bf L}} \,\, \text{mod} \,\, l^k.$$
Then, the elements $\{\Lambda_{m} (\xi_{w ({\bf L}), k})\}_{\bf L}$ form an Euler system in
$K$-theory without coefficients $\{K_{2m}(O_{{\bf L}l^k, S})_l\}_{\bf L}$. Namely, we have:

\begin{proposition}\label{Proposition ES1}
If ${\bf L'}={\bf l'L}$, then the following equality holds:
\begin{equation}
Tr_{F_{{\bf L}^{\prime} l^k}/F_{{\bf L} l^k}} (\Lambda_{m} (\xi_{w ({\bf L}^{\prime}), k})) =
\Lambda_{m} (N_{w ({\bf L}^{\prime})/w ({\bf L})}
( \xi_{w ({\bf L}^{\prime}), k}))^{1 - N({\bf l}^{\prime})^{m}
({\bf l}^{\prime}, \, F_{{\bf L} l^k})^{-1}}.
\end{equation}
\end{proposition}

\begin{proof}
The Proposition follows by (\ref{ESauxiliary1}) and
Lemma \ref{Lemma 2.1}.
\end{proof}
\bigskip

Let us fix an arbitrary integer $n>0$.  Next, we use the Euler system above to construct Euler systems $\{\blambda_{v({\bf L})}\}_{\bf L}$ in the
$K$--groups $\{K_{2n}(O_{{\bf L}, S})_l\}_{\bf L}$. The general idea is as follows. First, one constructs
Euler Systems $\{\blambda_{v({\bf L}), k}\}_{\bf L}$ in the $K$--theory with coefficients
$\{K_{2n}(O_{{\bf L}, S}, \Bbb Z/l^k)\}_{\bf L}$, for all $k>0$. Then one passes to a projective limit with respect to $k$.
The constructions, ideas and results developed in \S4 play a key role in what follows.
\medskip

\noindent
For every ${\bf L}$ as above and every $k\geq 0$,
we follow the ideas in \S4 and define the elements
$\blambda_{v({\bf L}), k} \in
K_{2n} ({\mathcal O}_{F_{{\bf L}}, S};\, \Z/l^k)$ by:
\begin{eqnarray}\label{EulerSystemElementsFiniteCoef}
\blambda_{v({\bf L}), k}:&=&
Tr_{F_{{\bf L}{l^{k}}} / F_{\bf L}} (x_{w_k({\bf L}), k}^{\Theta_{m} ({\bf b}_{\bf L},\, {\bf L}{\bf f}_k) } \ast
\beta_{k}^{\ast \, n-m})^{N {\bf b}_{\bf L}^{n-m} \, \gamma_{l}}=\\
\nonumber& & Tr_{F_{{\bf L}{l^{k}}} / F_{\bf L}} (\Lambda_m(\xi_{w_k({\bf L}),k}) \ast
\beta_{k}^{\ast \, n-m})^{N {\bf b}_{\bf L}^{n-m} \, \gamma_{l}},
\end{eqnarray}
where the operator $\gamma_l\in\Bbb Z_l[G(F/K)]$ is given in Definition \ref{gamma-l}.
The following theorem lies at the heart of our construction of the
Euler system for higher $K$-groups of CM abelian extensions of arbitrary totally real number fields.

\begin{theorem}\label{Theorem ES2}
For every $k^{\prime} \geq k$ and every
${\bf L}$ and ${\bf L}^{\prime} =
{\bf L} {\bf l}^{\prime}$ we have:
\begin{eqnarray}
\nonumber r_{k^{\prime}/k} (\blambda_{v({\bf L}), k'}) &=&
\blambda_{v({\bf L}), k}\\
\nonumber \partial_{F_{{\bf L}}} (\blambda_{v({\bf L}), k}) &=&
(N_{{\bf L}}(\xi_{w ({\bf L}), k} \ast \beta_{k}^{\ast \, n-m}))^{\Theta_{n} ({\bf b_{\bf L}}, {\bf f L})}\\
\nonumber \qquad Tr_{F_{{\bf L}^{\prime}}/F_{\bf L}}
(\blambda_{v ({\bf L}^{\prime}), k}) &=&
(\blambda'_{v({\bf L}), k})^{1 - N({\bf l}^{\prime})^{n}
({\bf l}^{\prime}, \, F_{{\bf L}})^{-1}},
\end{eqnarray}
where $N_{{\bf L}} := Tr_{k_{w({\bf L})}/k_{v({\bf L})}}$ and $\blambda'_{v({\bf L}), k}$ is defined by
$$\blambda'_{v({\bf L}), k}:=Tr_{F_{{\bf L}{l^{k}}} / F_{\bf L}} (\Lambda_m(N_{w_k({\bf L'})/w_k({\bf L})}(\xi_{w_k({\bf L'}),k})) \ast
\beta_{k}^{\ast \, n-m})^{N {\bf b}_{\bf L}^{n-m} \, \gamma_{l}}.$$
\end{theorem}

\begin{proof} The first formula
follows by Lemma \ref{Lemma 7.6}. The second
formula follows by Theorem \ref{Theorem 7.7}.
Let us prove the Euler System property (the third formula in the statement of the Theorem.)
We apply Lemma \ref{Lemma 2.1} and definition \eqref{EulerSystemElementsFiniteCoef}:

\begin{eqnarray}
\nonumber & Tr_{F_{{\bf L}^{\prime}}/F_{{\bf L}}}
(\blambda_{v({\bf L'}), k}) =
Tr_{F_{{\bf L}^{\prime}}/F_{{\bf L}}}
Tr_{F_{{{\bf L}^{\prime} l^k}}/ F_{{\bf L}^{\prime}}}
(x_{w ({\bf L}^{\prime}), k}^{\Theta_{m} ( {\bf b}_{{\bf L}^{\prime}},\,
{\bf f}_k {\bf L}^{\prime})} \ast \beta_{k}^{\ast \, n-m})^{N {\bf b}^{n-m} \gamma_l} = \\
\nonumber  & Tr_{F_{{{\bf L} l^k}} /F_{{\bf L}}}
Tr_{F_{{\bf L}^{\prime} l^k}/F_{{\bf L} l^k}}
(x_{w ({\bf L}^{\prime}), k}^{\Theta_{m} ( {\bf b}_{{\bf L}^{\prime}},\,
{\bf f}_k {\bf L}^{\prime})} \ast \beta_{k}^{\ast \, n-m})^{N {\bf b}^{n-m} \gamma_l}
=\\
\nonumber & Tr_{F_{{{\bf L} l^k}} /F_{{\bf L}}} (Tr_{F_{{\bf L}^{\prime} l^k}/F_{{\bf L} l^k}}
x_{w ({\bf L}^{\prime}), k}^{\Theta_{m} ( {\bf b}_{{\bf L}^{\prime}},\,
{\bf f}_k {\bf L}^{\prime})} \ast \beta_{k}^{\ast \, n-m})^{N {\bf b}^{n-m} \gamma_l}
=\\
\nonumber & Tr_{F_{{{\bf L} l^k}} /F_{{\bf L}}} (Tr_{F_{{\bf L}^{\prime} l^k}/F_{{\bf L} l^k}}
\bigl( x_{w ({\bf L}^{\prime}), k}\bigr)^{\Theta_{m} ( {\bf b}_{{\bf L}^{\prime}},\,
{\bf f}_k {\bf L}^{\prime})} \ast \beta_{k}^{\ast \, n-m})^{N {\bf b}^{n-m} \gamma_l}
=\\
\nonumber & Tr_{F_{{{\bf L} l^k}} /F_{{\bf L}}} (Tr_{F_{{\bf L}^{\prime} l^k}/F_{{\bf L} l^k}}
\bigl( x_{w ({\bf L}^{\prime}), k}\bigr)^{
Res_{K_{{\bf f}_k {\bf L}^{\prime}}/K_{{\bf f}_k {\bf L}}}\Theta_{m}
( {\bf b}_{{\bf L}^{\prime}},\,
{\bf f}_k {\bf L}^{\prime})} \ast \beta_{k}^{\ast \, n-m})^{N {\bf b}^{n-m} \gamma_l}
=\\
\nonumber & Tr_{F_{{{\bf L} l^k}} /F_{{\bf L}}} (Tr_{F_{{\bf L}^{\prime} l^k}/F_{{\bf L} l^k}}
\bigl( x_{w ({\bf L}^{\prime}), k}\bigr)^{
\bigl(1 - ({\bf l}^{\prime}, \, F_{{\bf L} l^k} )^{-1} N({\bf l'})^{m} \, \bigr) \,\, \Theta_{m}
( {\bf b}_{{\bf L}},\,
{\bf f}_k {\bf L})} \ast \beta_{k}^{\ast \, n-m})^{N {\bf b}^{n-m} \gamma_l}
=\\
\nonumber & Tr_{F_{{{\bf L} l^k}} /F_{{\bf L}}} (Tr_{F_{{\bf L}^{\prime} l^k}/F_{{\bf L} l^k}}
\bigl( x_{w ({\bf L}^{\prime}), k}\bigr)^{ \,\, \Theta_{m} ( {\bf b}_{{\bf L}},\,
{\bf f}_k {\bf L})} \ast \beta_{k}^{\ast \, n-m})^{N {\bf b_{\bf L}}^{n-m} \gamma_l
\bigl(1 - ({\bf l}^{\prime}, \, F_{\bf L})^{-1} N({\bf l'})^{n}) \, \bigr)}
= \\
&(\blambda'_{v({\bf L}), k})^{1 - N({\bf l}^{\prime})^{n}
({\bf l}^{\prime}, \, F_{{\bf L}})^{-1}}.
\nonumber
\end{eqnarray}
\medskip

\noindent
The last equality is a direct consequence of equalities
(\ref{ESauxiliary1}) and (\ref{ESauxiliary2}).
\end{proof}
\medskip

\noindent
Now, let ${\bf b}$ be a fixed ideal in $O_K$, coprime to ${\bf f} \, l.$
Consider all ${\bf L}$ as above which are coprime to $l {\bf b} {\bf f}.$ Naturally, we can choose
${\bf b}_{{\bf L}} := {\bf b}$, for all such ${{\bf L}}.$
These choices and the results of \S4 permit us to define the elements $\blambda_{v({\bf L})}\in K_{2n} ({\mathcal O}_{F_{\bf L}, S})_l$
and $\xi_{v({\bf L})} \in K_{2n-1} (k_{v({\bf L})})_l$ as follows.

\begin{definition}\label{ES1-definition} Let $\blambda_{v({\bf L})}
\in K_{2n} ({\mathcal O}_{F_{\bf L}, S})_l$
be the element corresponding to
$$(\blambda_{v({\bf L}), l^k})_k\in
\varprojlim_{k} K_{2n} ({\mathcal O}_{F_{\bf L}, S} ;\, \Z/l^k)$$
via the isomorphism (\ref{invlimringofintegers})
for the ring ${\mathcal O}_{F_{\bf L}, S}$.
\end{definition}

\begin{definition}\label{ES-definition} Let $\xi_{v({\bf L})} \in K_{2n-1} (k_{v({\bf L})})_l$
be the element corresponding to
$$(N_{\bf L}(\xi_{w({\bf L}), k} \ast \beta_{k}^{\ast \, n-m}))_k \in
\varprojlim_{k} K_{2n-1} (k_{v ({\bf L})} ;\, \Z/l^k)$$
via the isomorphism (\ref{invlimfinitefields}) for the finite field
$k_{v ({\bf L})}.$
\end{definition}

\begin{remark} Note that $\Lambda(\xi_{v({\bf L})})=\blambda_{v({\bf L})}$ (see Remark \ref{xi-lambda}.)
\end{remark}

The next result shows that the elements
$\{\blambda_{v ({\bf L})}\}_{\bf L}$ provide an Euler System for the $K$-theory
without coefficients $\{K_{2n}(O_{F_{\bf L}, S})_l\}_{\bf L}$.

\begin{theorem}\label{Theorem ES3}
For every ${\bf L}$ and ${\bf L}^{\prime}$ as above, such that
${\bf L}^{\prime} = {\bf L} {\bf l}^{\prime}$, we have the following equalities:
\begin{equation}
\partial_{F_{{\bf L}}} (\Lambda (\xi_{v ({\bf L})})) =
\xi_{v ({\bf L})}^{\Theta_{n} ({\bf b}, \, {\bf  L f})}
\label{ESBoundaryPropertyNoCoef}
\end{equation}

\begin{equation}
Tr_{F_{{\bf L}^{\prime}}/F_{\bf L}}
(\Lambda (\xi_{v ({\bf L}^{\prime})})) =
\Lambda (N_{v ({\bf L}^{\prime})/v ({\bf L})}
(\xi_{v ({\bf L}^{\prime})}))^{1 - N({\bf l}^{\prime})^{n} ({\bf l}^{\prime}, \, F_{\bf L})}
\label{ESPropertyNoCoef}
\end{equation}
\end{theorem}

\begin{proof} This follows directly from Theorem \ref{Theorem ES2}.
\end{proof}

\begin{remark} \label{EtaleES}
It is easy to see that one can construct Euler systems
for \'etale $K$-theory in a similar manner.
 \end{remark}

\begin{remark} In our upcoming work, we will use the Euler systems constructed above to investigate the structure of the group
of divisible elements $div K_{2n}(F)_l$ inside $K_{2n} (F)_l.$ The structure of  $div K_{2n}(F)_l$
is of principal interest vis a vis some classical conjectures in algebraic number theory, as explained in
the introduction.
\end{remark}


\medskip

\noindent
{\it Acknowledgments}:\quad
The first author would like to thank the University
of California, San Diego for hospitality and financial
support during his stay in April 2009, when this collaboration began, and during the period
December 2010 -- June 2011.
Also, he thanks the SFB in Muenster for hospitality and financial
support during his stay in September 2009 and the Max Planck Institute in Bonn
for hospitality and financial support during his stay in April and May 2010.
\bigskip


\bibliographystyle{amsplain}

\end{document}